\documentclass{article}
\usepackage{amssymb}
\usepackage{amsmath}
\usepackage{amsfonts}

\newtheorem{theorem}{Theorem}[section]

\newtheorem{corollary}{Corollary}
\newtheorem{criterion}{Criterion}[section]
\newtheorem{definition}{Definition}[section]

\newtheorem{proposition}{Proposition}[section]

\newenvironment{proof}[1][Proof]{\noindent\textbf{#1.} }{\ \rule{0.5em}{0.5em}}

\input{tcilatex}
\begin{document}

\title{{\Large AUTOMORPHIC EQUIVALENCE OF THE REPRESENTATIONS OF LIE
ALGEBRAS.}}
\author{{\Large I. Shestakov,} \\
shestak@ime.usp.br \and {\Large A.Tsurkov} \\
arkady.tsurkov@gmail.com\\
Institute of Mathematics and Statistics.\\
University S\~{a}o Paulo. \\
Rua do Mat\~{a}o, 1010 \\
Cidade Universit\'{a}ria \\
S\~{a}o Paulo - SP - Brasil - CEP 05508-090 }
\maketitle

\begin{abstract}
In this paper we research the algebraic geometry of the representations of
Lie algebras over fixed field $k$. We assume that this field is infinite and 
$char\left( k\right) =0$. We consider the representations of Lie algebras as 
$2$-sorted universal algebras. The representations of groups were considered
by similar approach: as $2$-sorted universal algebras - in \cite%
{PlotkinVovsi} and \cite{Repofgr}. The basic notions of the algebraic
geometry of representations of Lie algebras we define similar to the basic
notions of the algebraic geometry of representations of groups (see \cite%
{Repofgr}). We prove that if a field $k$ has not nontrivial automorphisms
then automorphic equivalence of representations of Lie algebras coincide
with geometric equivalence. This result is similar to the result of \cite%
{PlotkinZhitAutCat}, which was achieved for representations of groups. But
we achieve our result by another method: by consideration of $1$-sorted
objects. We suppose that our method can be more perspective in the further
researches.
\end{abstract}

\section{Introduction: representations of Lie algebras as $2$-sorted
universal algebras.\label{Intro}}

\setcounter{equation}{0}

In this paper we research the algebraic geometry of the representations of
Lie algebras.

We consider the Lie algebras over the field $k$. And we say that we have the
representation of Lie algebra $\left( L,V\right) $ if the elements of the
Lie algebra $L$ act on the vector space $V$ over the field $k$ as linear
transformations and the mapping $\mathfrak{f}:L\rightarrow \mathrm{End}%
_{k}\left( V\right) $ which we define by $\mathfrak{f}\left( l\right) \left(
v\right) =l\circ v$, where $l\in L$, $v\in V$, $\circ $ is acting of the
elements of the algebra $L$ over elements of $V$, is a homomorphism from the
Lie algebra $L$ to the Lie algebra $\mathrm{End}_{k}^{\left( -\right)
}\left( V\right) =\mathfrak{end}_{k}\left( V\right) $. Some time we will
omit the symbol $\circ $. In this paper we assume that $k$ is infinite and $%
char\left( k\right) =0$.

We consider a representation of Lie algebra as $2$-sorted universal algebra.
Particularly the homomorphisms of representations we define by this
definition:

\begin{definition}
\label{homomorphisms}We say that we have a \textbf{homomorphism} $\left(
\varphi ,\psi \right) $ from the representation $\left( L_{1},V_{1}\right) $
to the representation $\left( L_{2},V_{2}\right) $ if we have a homomorphism
of Lie algebras $\varphi :L_{1}\rightarrow L_{2}$ and a linear map $\psi
:V_{1}\rightarrow V_{2}$ such that 
\begin{equation}
\varphi \left( l\right) \circ \psi \left( v\right) =\psi \left( l\circ
v\right)  \label{homomorphism_formula}
\end{equation}
holds for every $l\in L_{1}$ and every $v\in V_{1}$.
\end{definition}

We denote $\left( \varphi ,\psi \right) :\left( L_{1},V_{1}\right)
\rightarrow \left( L_{2},V_{2}\right) $.

It means that the field $k$ is fixed in our considerations. But algebras Lie
and theirs modules me can change and we can compare the algebraic geometry
of representations $\left( L_{1},V_{1}\right) $ and $\left(
L_{2},V_{2}\right) $ such that $L_{1}\neq L_{2}$ and $V_{1}\neq V_{2}$.
Therefore the multiplication by scalars of the elements of the algebra Lie $%
L $ and the elements of its module $V$ we can consider as unary operations:
for every scalar $\lambda \in k$ we have two unary operations. But the
acting of the elements of the algebra Lie $L$ over the elements of its
module $V$ we must consider as one binary $2$-sorted operation.

If $\left( \varphi ,\psi \right) :\left( L,V\right) \rightarrow \left(
P,Q\right) $ is a homomorphism of the representations, than $\ker \varphi $
is an ideal of the Lie algebra $L$, $\ker \psi $ is a $L$-submodule of the $%
L $-module $V$, $\left( \ker \varphi ,\ker \psi \right) $ is a
representation and a congruence in $\left( L,V\right) $.

If $H=\left( L,V\right) $ is a representation of Lie algebra and $%
T_{1}\subseteq L$, $T_{2}\subseteq V$ we will denote $\left(
T_{1},T_{2}\right) \subseteq H$. If also $P_{1}\subseteq L$, $P_{2}\subseteq
V$ we will denote $\left( T_{1},T_{2}\right) \cap \left( P_{1},P_{2}\right)
=\left( T_{1}\cap P_{1},T_{2}\cap P_{2}\right) $.

\section{Basic notions of the algebraic geometry of representations of Lie
algebras.}

\setcounter{equation}{0}

We denote by $\Xi $ the variety of the all representations of Lie algebras
over the fixed field $k$.

\begin{definition}
\label{freerep}We say that the representation $\left( L,V\right) $ is a 
\textbf{free representation} with the pair of sets of the free generators $%
\left\{ X,Y\right\} $ if $X\subset L$, $Y\subset V$ and for every
representation $\left( P,U\right) $ every pair of mappings $\varphi
:X\rightarrow P$, $\psi :Y\rightarrow U$ can by extended to homomorphism $%
\left( \varphi ,\psi \right) :\left( L,V\right) \rightarrow \left(
P,U\right) $.
\end{definition}

We will denote this representation by $W=W\left( X,Y\right) $. It is well
known that $W\left( X,Y\right) =\left( L\left( X\right) ,A\left( X\right)
Y\right) $, where $L\left( X\right) =L$ is the free Lie algebra with the set 
$X$ of free generators, $A\left( X\right) $ is the free associative algebra
with unit which has the set $X$ of free generators, $A\left( X\right)
Y=\bigoplus\limits_{y\in Y}A\left( X\right) y=V$ is the free $A\left(
X\right) $ module with the basis $Y$. In this notation the symbol $\circ $
of the action is omitted. In particular, if $X=\varnothing $ then $L\left(
X\right) =\left\{ 0\right\} $, $A\left( X\right) =k$, if $Y=\varnothing $
then $A\left( X\right) Y=\left\{ 0\right\} $, if $X=\left\{ x\right\} $ then 
$L\left( X\right) =kx$, $A\left( X\right) =k\left[ x\right] $.

$X^{0}$, $Y^{0}$ will be infinite countable sets of symbols. We consider the
category $\Xi ^{0}$. $\mathrm{Ob}\Xi ^{0}=\left\{ W\left( X,Y\right) \mid
\left\vert X\right\vert <\infty ,\left\vert Y\right\vert <\infty ,X\subset
X^{0},Y\subset Y^{0}\right\} $. Morphisms of this category are homomorphisms
of its objects. The category $\Xi ^{0}$ is a small category: $\mathrm{Ob}\Xi
^{0}$ and $\mathrm{Mor}\Xi ^{0}$ are sets. So we can tell about elements and
subsets of $\mathrm{Ob}\Xi ^{0}$ and $\mathrm{Mor}\Xi ^{0}$.

We will take our equations from the representations $W=W\left( X,Y\right)
=\left( L\left( X\right) ,A\left( X\right) Y\right) \in \mathrm{Ob}\Xi ^{0}$%
. We have two sorts of equations: the equations in the Lie algebra - $%
t_{1}\in L\left( X\right) $ and the action type equations - $t_{2}\in
A\left( X\right) Y$. We can resolve our equations in arbitrary $H=\left(
L,V\right) \in \Xi $. The homomorphism $\left( \varphi ,\psi \right)
:W\left( X,Y\right) \rightarrow H$ will be the solution of the equation $%
t_{1}\in L\left( X\right) $ if $\varphi \left( t_{1}\right) =0$ and will be
the solution of the equation $t_{2}\in A\left( X\right) Y$ if $\psi \left(
t_{2}\right) =0$.

We can consider the system of equations $T=\left( T_{1},T_{2}\right) $,
where $T_{1}\subseteq L\left( X\right) $, $T_{2}\subseteq A\left( X\right) Y$%
. We can consider this system as a set $T=T_{1}\cup T_{2}$ but it is not
natural because the subsets $T_{1}$ and $T_{2}$ have different origins: $%
T_{1}\subseteq L\left( X\right) $, $T_{2}\subseteq A\left( X\right) Y$. So
it is natural to consider the system of equations $T=\left(
T_{1},T_{2}\right) $ as a pair of sets. However for the sake of brevity we
will some time write "the set $\left( T_{1},T_{2}\right) $". The set of
solutions of the system $\left( T_{1},T_{2}\right) $ in the representation $%
H=\left( L,V\right) $ is%
\begin{equation*}
\left( T_{1},T_{2}\right) _{H}^{\prime }=\left\{ \left( \varphi ,\psi
\right) \in \mathrm{Hom}\left( W\left( X,Y\right) ,H\right) \mid
T_{1}\subseteq \ker \varphi ,T_{2}\subseteq \ker \psi \right\} .
\end{equation*}%
Vice versa, for every set $A\subset $ $\mathrm{Hom}\left( W\left( X,Y\right)
,H\right) $ we can consider the set 
\begin{equation*}
A_{H}^{\prime }=\left( \bigcap\limits_{\left( \varphi ,\psi \right) \in
A}\ker \varphi ,\bigcap\limits_{\left( \varphi ,\psi \right) \in A}\ker \psi
\right) .
\end{equation*}%
This set will be the maximal system of equations, such that $A$ is a subset
of the set of its solutions. Also we can consider the algebraic closer of
the system $\left( T_{1},T_{2}\right) $: 
\begin{equation*}
\left( T_{1},T_{2}\right) _{H}^{\prime \prime }=\left(
\bigcap\limits_{\left( \varphi ,\psi \right) \in \left( T_{1},T_{2}\right)
_{H}^{\prime }}\ker \varphi ,\bigcap\limits_{\left( \varphi ,\psi \right)
\in \left( T_{1},T_{2}\right) _{H}^{\prime }}\ker \psi \right) =
\end{equation*}%
\begin{equation*}
\left( \bigcap\limits_{\left( \varphi ,\psi \right) \in \mathrm{Hom}\left(
W,H\right) ,T_{1}\subseteq \ker \varphi ,T_{2}\subseteq \ker \psi }\ker
\varphi ,\bigcap\limits_{\left( \varphi ,\psi \right) \in \mathrm{Hom}\left(
W,H\right) ,T_{1}\subseteq \ker \varphi ,T_{2}\subseteq \ker \psi }\ker \psi
\right) .
\end{equation*}%
It will be the maximal system of equations which have the same solutions as $%
\left( T_{1},T_{2}\right) $.

It is clear that $\left( T_{1},T_{2}\right) \subseteq \left(
T_{1},T_{2}\right) _{H}^{\prime \prime }$ holds for every $W\left(
X,Y\right) \in \mathrm{Ob}\Xi ^{0}$, every $\left( T_{1},T_{2}\right)
\subseteq W\left( X,Y\right) $ and every $H\in \Xi $.

\begin{definition}
The set $\left( T_{1},T_{2}\right) \subseteq W\left( X,Y\right) $ is $H$%
\textbf{-closed} if $\left( T_{1},T_{2}\right) _{H}^{\prime \prime }=\left(
T_{1},T_{2}\right) $.
\end{definition}

It is clear that the closed sets are congruences. The family of the all $H$%
-closed sets in the free representation $W=W\left( X,Y\right) \in \mathrm{Ob}%
\Xi ^{0}$ we denote by $Cl_{H}\left( W\right) $.

\begin{definition}
$H_{1},H_{2}\in \Xi $. $H_{1},H_{2}$ are called \textbf{geometrically
equivalent} if $\left( T_{1},T_{2}\right) _{H_{1}}^{\prime \prime }=\left(
T_{1},T_{2}\right) _{H_{2}}^{\prime \prime }$ holds for every $W\left(
X,Y\right) \in \mathrm{Ob}\Xi ^{0}$ and every $\left( T_{1},T_{2}\right)
\subseteq W\left( X,Y\right) $.
\end{definition}

We consider $W_{1}=W\left( X_{1},Y_{1}\right) ,W_{2}=W\left(
X_{2},Y_{2}\right) \in \mathrm{Ob}\Xi ^{0}$ and $\left( T_{1},T_{2}\right) $
some congruence in $W_{2}$. We denote by $\beta =\beta _{W_{1},W_{2}}\left(
T_{1},T_{2}\right) $ the following relation in $\mathrm{Hom}\left(
W_{1},W_{2}\right) $: $\left( \left( \varphi _{1},\psi _{1}\right) ,\left(
\varphi _{2},\psi _{2}\right) \right) \in \beta $ if and only if $\varphi
_{1}\left( l\right) \equiv \varphi _{2}\left( l\right) \left( \func{mod}%
T_{1}\right) $ holds for every $l\in L\left( X_{1}\right) $ and $\psi
_{1}\left( v\right) \equiv \psi _{2}\left( v\right) \left( \func{mod}%
T_{2}\right) $ holds for every $v\in A\left( X_{1}\right) Y_{1}$. This
relation is a $2$-sorted analog of the relation $\beta $ from \cite[%
Subsection 3.3]{PlotkinSame}. Now we define as in \cite[Subsection 3.4]%
{PlotkinSame}

\begin{definition}
$H_{1},H_{2}\in \Xi $. $H_{1},H_{2}$ are called \textbf{automorphically
equivalent} if these 3 conditions hold:
\end{definition}

\begin{enumerate}
\item \textit{There exists an automorphism }$\Phi :\Xi ^{0}\rightarrow \Xi
^{0}$\textit{.}

\item \textit{There exists a function }$\alpha =\alpha \left( \Phi \right) $%
\textit{\ such that }$\alpha \left( \Phi \right) _{W}:Cl_{H_{1}}\left(
W\right) \rightarrow Cl_{H_{2}}\left( \Phi \left( W\right) \right) $\textit{%
\ is a bijection for every }$W\in \mathrm{Ob}\Xi ^{0}$\textit{.}

\item $\Phi \left( \beta _{W_{1},W_{2}}\left( T_{1},T_{2}\right) \right)
=\beta _{\Phi \left( W_{1}\right) ,\Phi \left( W_{2}\right) }\left( \alpha
\left( \Phi \right) _{W_{2}}\left( T_{1},T_{2}\right) \right) $\textit{\
holds for every }$W_{1},W_{2}\in \mathrm{Ob}\Xi ^{0}$\textit{, and every }$%
\left( T_{1},T_{2}\right) \in $ $Cl_{H_{1}}\left( W_{2}\right) $.
\end{enumerate}

Here $\Phi \left( \left( \varphi _{1},\psi _{1}\right) ,\left( \varphi
_{2},\psi _{2}\right) \right) =\left( \Phi \left( \varphi _{1},\psi
_{1}\right) ,\Phi \left( \varphi _{2},\psi _{2}\right) \right) $.

It can be proved as in \cite[Proposition 8]{PlotkinSame} that if $H_{1}$ and 
$H_{2}$ are automorphically equivalent then function $\alpha $ is uniquely
determined by automorphism $\Phi $.

\section{Some facts about the closed congruences in the free representations
of Lie algebras.}

\setcounter{equation}{0}

In this Section we assume that $X_{1}\subseteq X_{2}\subset X^{0}$, $%
Y_{1}\subseteq Y_{2}\subset Y^{0}$, $\left( L,V\right) =H\in \Xi $. We
denote $\left( L\left( X_{i}\right) ,A\left( X_{i}\right) Y_{i}\right)
=W\left( X_{i},Y_{i}\right) =W_{i}$, where $i=1,2$.

If $\left( T_{1},T_{2}\right) \subseteq W_{1}$, then, because $%
W_{1}\subseteq W_{2}$, we can consider the sets 
\begin{equation*}
\left( T_{1},T_{2}\right) _{W_{i},H}^{\prime }=\left\{ \left( \varphi ,\psi
\right) :W_{i}\rightarrow H\mid T_{1}\subseteq \ker \varphi ,T_{2}\subseteq
\ker \psi \right\}
\end{equation*}%
and we will denote $\left( \left( T_{1},T_{2}\right) _{W_{i},H}^{\prime
}\right) _{H}^{\prime }=\left( T_{1},T_{2}\right) _{W_{i},H}^{\prime \prime
} $, where $i=1,2$. We say that $\left( T_{1},T_{2}\right) $ is $H$-closed
in $W_{i}$ if $\left( T_{1},T_{2}\right) _{W_{i},H}^{\prime \prime }=\left(
T_{1},T_{2}\right) $. In all other sections of this paper it is clear what
kind of algebraic closer of the system of equations we consider. But in this
Section we must fine distinguish between the different features.

\begin{proposition}
\label{hom_for_cl}We assume that $\left( T_{1},T_{2}\right) \subseteq W_{2}$%
, $\left( \mu ,\nu \right) \in \left( T_{1}\cap L\left( X_{1}\right)
,T_{2}\cap A\left( X_{1}\right) Y_{1}\right) _{W_{1},H}^{\prime }$. We
denote 
\begin{equation*}
\left[ \left( \mu ,\nu \right) \right] =\left\{ \left( \varphi ,\psi \right)
\in \left( T_{1},T_{2}\right) _{H}^{\prime }\mid \varphi _{\mid X_{1}}=\mu
_{\mid X_{1}},\psi _{\mid Y_{1}}=\nu _{\mid Y_{1}}\right\} .
\end{equation*}%
Then 
\begin{equation*}
\left( \left( \bigcap\limits_{\left( \varphi ,\psi \right) \in \left[ \left(
\mu ,\nu \right) \right] }\ker \varphi \right) \cap L\left( X_{1}\right)
,\left( \bigcap\limits_{\left( \varphi ,\psi \right) \in \left[ \left( \mu
,\nu \right) \right] }\ker \psi \right) \cap A\left( X_{1}\right)
Y_{1}\right) =\left( \ker \mu ,\ker \nu \right)
\end{equation*}%
holds.
\end{proposition}

\begin{proof}
If $t_{1}\in \left( \bigcap\limits_{\left( \varphi ,\psi \right) \in \left[
\left( \mu ,\nu \right) \right] }\ker \varphi \right) \cap L\left(
X_{1}\right) $, then $\mu \left( t_{1}\right) =\varphi \left( t_{1}\right)
=0 $ for every $\varphi $ such that $\left( \varphi ,\psi \right) \in \left[
\left( \mu ,\nu \right) \right] $. If $t_{2}\in \left(
\bigcap\limits_{\left( \varphi ,\psi \right) \in \left[ \left( \mu ,\nu
\right) \right] }\ker \psi \right) \cap A\left( X_{1}\right) Y_{1}$, then $%
\nu \left( t_{2}\right) =\psi \left( t_{2}\right) =0$ for every $\psi $ such
that $\left( \varphi ,\psi \right) \in \left[ \left( \mu ,\nu \right) \right]
$.

If $t_{1}\in \ker \mu $, then $t_{1}\in L\left( X_{1}\right) $, so $\varphi
\left( t_{1}\right) =\mu \left( t_{1}\right) =0$ holds for every $\varphi $
such that $\left( \varphi ,\psi \right) \in \left[ \left( \mu ,\nu \right) %
\right] $. If $t_{2}\in \ker \nu $, then $t_{2}\in A\left( X_{1}\right)
Y_{1} $, so $\psi \left( t_{2}\right) =\nu \left( t_{2}\right) =0$ holds for
every $\psi $ such that $\left( \varphi ,\psi \right) \in \left[ \left( \mu
,\nu \right) \right] $.
\end{proof}

\begin{proposition}
\label{cl1}If $\left( T_{1},T_{2}\right) \subseteq W_{2}$ is $H$-closed,
then 
\begin{equation*}
\left( T_{1},T_{2}\right) \cap W_{1}=\left( T_{1}\cap L\left( X_{1}\right)
,T_{2}\cap A\left( X_{1}\right) Y_{1}\right)
\end{equation*}%
is $H$-closed in $W_{1}$.
\end{proposition}

\begin{proof}
\begin{equation*}
\left( T_{1},T_{2}\right) _{H}^{\prime \prime }=\left(
\bigcap\limits_{\left( \varphi ,\psi \right) \in \left( T_{1},T_{2}\right)
_{H}^{\prime }}\ker \varphi ,\bigcap\limits_{\left( \varphi ,\psi \right)
\in \left( T_{1},T_{2}\right) _{H}^{\prime }}\ker \psi \right) =\left(
T_{1},T_{2}\right) .
\end{equation*}%
\begin{equation*}
\left( T_{1}\cap L\left( X_{1}\right) ,T_{2}\cap A\left( X_{1}\right)
Y_{1}\right) _{W_{1},H}^{\prime \prime }=
\end{equation*}%
\begin{equation*}
\left( \bigcap\limits_{\left( \mu ,\nu \right) \in \left( T_{1}\cap L\left(
X_{1}\right) ,T_{2}\cap A\left( X_{1}\right) Y_{1}\right) _{W_{1},H}^{\prime
}}\ker \mu ,\bigcap\limits_{\left( \mu ,\nu \right) \in \left( T_{1}\cap
L\left( X_{1}\right) ,T_{2}\cap A\left( X_{1}\right) Y_{1}\right)
_{W_{1},H}^{\prime }}\ker \nu \right) .
\end{equation*}

We will consider $\left( \varphi ,\psi \right) \in \left( T_{1},T_{2}\right)
_{H}^{\prime }$. There exists only one $\left( \mu ,\nu \right) \in \mathrm{%
Hom}\left( W_{1},H\right) $ such that $\varphi _{\mid X_{1}}=\mu _{\mid
X_{1}}$, $\psi _{\mid Y_{1}}=\nu _{\mid Y_{1}}$. If $t_{1}\in T_{1}\cap
L\left( X_{1}\right) $, then $\mu \left( t_{1}\right) =\varphi \left(
t_{1}\right) =0$, if $t_{2}\in T_{2}\cap A\left( X_{1}\right) Y_{1}$, then $%
\nu \left( t_{2}\right) =\psi \left( t_{2}\right) =0$. Hence $\left( \mu
,\nu \right) \in \left( T_{1}\cap L\left( X_{1}\right) ,T_{2}\cap A\left(
X_{1}\right) Y_{1}\right) _{W_{1},H}^{\prime }$. So by Proposition \ref%
{hom_for_cl} 
\begin{equation*}
\left( \left( \bigcap\limits_{\left( \varphi ,\psi \right) \in \left[ \left(
\mu ,\nu \right) \right] }\ker \varphi \right) \cap L\left( X_{1}\right)
,\left( \bigcap\limits_{\left( \varphi ,\psi \right) \in \left[ \left( \mu
,\nu \right) \right] }\ker \psi \right) \cap A\left( X_{1}\right)
Y_{1}\right) =\left( \ker \mu ,\ker \nu \right) .
\end{equation*}

The set $\left( T_{1},T_{2}\right) _{H}^{\prime }$ can by presented as union
of the disjoint sets $\left[ \left( \mu ,\nu \right) \right] $, where $%
\left( \mu ,\nu \right) \in \mathrm{Hom}\left( W_{1},H\right) $ such that
exists $\left( \varphi ,\psi \right) \in \left( T_{1},T_{2}\right)
_{H}^{\prime }$, for which $\left( \varphi ,\psi \right) \in \left[ \left(
\mu ,\nu \right) \right] $ holds. 
\begin{equation*}
\left( T_{1}\cap L\left( X_{1}\right) ,T_{2}\cap A\left( X_{1}\right)
Y_{1}\right) _{W_{1},H}^{\prime \prime }\subseteq
\end{equation*}%
\begin{equation*}
\left( \bigcap\limits_{\substack{ \left( \mu ,\nu \right) \in \left(
T_{1}\cap L\left( X_{1}\right) ,T_{2}\cap A\left( X_{1}\right) Y_{1}\right)
_{H}^{\prime },  \\ \exists \left( \varphi ,\psi \right) \in \left(
T_{1},T_{2}\right) _{H}^{\prime }\mid \left( \varphi ,\psi \right) \in \left[
\left( \mu ,\nu \right) \right] }}\ker \mu ,\bigcap\limits_{\substack{ %
\left( \mu ,\nu \right) \in \left( T_{1}\cap L\left( X_{1}\right) ,T_{2}\cap
A\left( X_{1}\right) Y_{1}\right) _{H}^{\prime },  \\ \exists \left( \varphi
,\psi \right) \in \left( T_{1},T_{2}\right) _{H}^{\prime }\mid \left(
\varphi ,\psi \right) \in \left[ \left( \mu ,\nu \right) \right] }}\ker \nu
\right) .
\end{equation*}%
\begin{equation*}
\bigcap\limits_{\substack{ \left( \mu ,\nu \right) \in \left( T_{1}\cap
L\left( X_{1}\right) ,T_{2}\cap A\left( X_{1}\right) Y_{1}\right)
_{H}^{\prime },  \\ \exists \left( \varphi ,\psi \right) \in \left(
T_{1},T_{2}\right) _{H}^{\prime }\mid \left( \varphi ,\psi \right) \in \left[
\left( \mu ,\nu \right) \right] }}\ker \mu =
\end{equation*}%
\begin{equation*}
\bigcap\limits_{\substack{ \left( \mu ,\nu \right) \in \left( T_{1}\cap
L\left( X_{1}\right) ,T_{2}\cap A\left( X_{1}\right) Y_{1}\right)
_{H}^{\prime },  \\ \exists \left( \varphi ,\psi \right) \in \left(
T_{1},T_{2}\right) _{H}^{\prime }\mid \left( \varphi ,\psi \right) \in \left[
\left( \mu ,\nu \right) \right] }}\left( \left( \bigcap\limits_{\left(
\varphi ,\psi \right) \in \left[ \left( \mu ,\nu \right) \right] }\ker
\varphi \right) \cap L\left( X_{1}\right) \right) =
\end{equation*}%
\begin{equation*}
\left( \bigcap\limits_{\left( \varphi ,\psi \right) \in \left(
T_{1},T_{2}\right) _{H}^{\prime }}\ker \varphi \right) \cap L\left(
X_{1}\right) =T_{1}\cap L\left( X_{1}\right) .
\end{equation*}%
\begin{equation*}
\bigcap\limits_{\substack{ \left( \mu ,\nu \right) \in \left( T_{1}\cap
L\left( X_{1}\right) ,T_{2}\cap A\left( X_{1}\right) Y_{1}\right)
_{H}^{\prime },  \\ \exists \left( \varphi ,\psi \right) \in \left(
T_{1},T_{2}\right) _{H}^{\prime }\mid \left( \varphi ,\psi \right) \in \left[
\left( \mu ,\nu \right) \right] }}\ker \nu =
\end{equation*}%
\begin{equation*}
\bigcap\limits_{\substack{ \left( \mu ,\nu \right) \in \left( T_{1}\cap
L\left( X_{1}\right) ,T_{2}\cap A\left( X_{1}\right) Y_{1}\right)
_{H}^{\prime },  \\ \exists \left( \varphi ,\psi \right) \in \left(
T_{1},T_{2}\right) _{H}^{\prime }\mid \left( \varphi ,\psi \right) \in \left[
\left( \mu ,\nu \right) \right] }}\left( \left( \bigcap\limits_{\left(
\varphi ,\psi \right) \in \left[ \left( \mu ,\nu \right) \right] }\ker \psi
\right) \cap A\left( X_{1}\right) Y_{1}\right) =
\end{equation*}%
\begin{equation*}
\left( \bigcap\limits_{\left( \varphi ,\psi \right) \in \left(
T_{1},T_{2}\right) _{H}^{\prime }}\ker \psi \right) \cap A\left(
X_{1}\right) Y_{1}=T_{2}\cap A\left( X_{1}\right) Y_{1}.
\end{equation*}
\end{proof}

\begin{proposition}
\label{cl2}If $\left( T_{1},T_{2}\right) \subseteq W_{1}$ is $H$-closed in $%
W_{1}$, then 
\begin{equation*}
\left( T_{1},T_{2}\right) =\left( T_{1},T_{2}\right) _{W_{2},H}^{\prime
\prime }\cap W_{1}.
\end{equation*}
\end{proposition}

\begin{proof}
In $\left( T_{1},T_{2}\right) _{W_{2},H}^{\prime }=\left\{ \left( \varphi
,\psi \right) \in \mathrm{Hom}\left( W_{2},H\right) \mid \ker \varphi
\supseteq T_{1},\ker \psi \supseteq T_{2}\right\} $ we can define
equivalence: $\left( \varphi _{1},\psi _{1}\right) \sim \left( \varphi
_{2},\psi _{2}\right) $ if and only if $\varphi _{1\mid X_{1}}=\varphi
_{2\mid X_{1}}$, $\psi _{1\mid Y_{1}}=\psi _{2\mid Y_{1}}$. As in the proof
of Proposition \ref{cl1}, for every class of this equivalence there exist
only one $\left( \mu ,\nu \right) \in \left( T_{1},T_{2}\right)
_{W_{1},H}^{\prime }$ such that this class coincide with $\left[ \left( \mu
,\nu \right) \right] $. Vice versa, for every $\left( \mu ,\nu \right) \in
\left( T_{1},T_{2}\right) _{W_{1},H}^{\prime }$ there exist only one class
of elements of the set $\left( T_{1},T_{2}\right) _{W_{2},H}^{\prime }$,
which coincide with $\left[ \left( \mu ,\nu \right) \right] $.%
\begin{equation*}
\left( T_{1},T_{2}\right) _{W_{2},H}^{\prime \prime }\cap W\left(
X_{1},Y_{1}\right) =
\end{equation*}%
\begin{equation*}
\left( \left( \bigcap\limits_{\left( \varphi ,\psi \right) \in \left(
T_{1},T_{2}\right) _{W_{2},H}^{\prime }}\ker \varphi \right) \cap L\left(
X_{1}\right) ,\left( \bigcap\limits_{\left( \varphi ,\psi \right) \in \left(
T_{1},T_{2}\right) _{W_{2},H}^{\prime }}\ker \psi \right) \cap A\left(
X_{1}\right) Y_{1}\right) .
\end{equation*}%
By Proposition \ref{hom_for_cl} we have%
\begin{equation*}
\left( \bigcap\limits_{\left( \varphi ,\psi \right) \in \left(
T_{1},T_{2}\right) _{W_{2},H}^{\prime }}\ker \varphi \right) \cap L\left(
X_{1}\right) =
\end{equation*}%
\begin{equation*}
\bigcap\limits_{\left( \mu ,\nu \right) \in \left( T_{1},T_{2}\right)
_{W_{1},H}^{\prime }}\left( \left( \bigcap\limits_{\left( \varphi ,\psi
\right) \in \left[ \left( \mu ,\nu \right) \right] }\ker \varphi \right)
\cap L\left( X_{1}\right) \right) =
\end{equation*}%
\begin{equation*}
\bigcap\limits_{\left( \mu ,\nu \right) \in \left( T_{1},T_{2}\right)
_{W_{1},H}^{\prime }}\ker \mu =T_{1}.
\end{equation*}%
\begin{equation*}
\left( \bigcap\limits_{\left( \varphi ,\psi \right) \in \left(
T_{1},T_{2}\right) _{W_{2},H}^{\prime }}\ker \psi \right) \cap A\left(
X_{1}\right) Y_{1}=
\end{equation*}%
\begin{equation*}
\bigcap\limits_{\left( \mu ,\nu \right) \in \left( T_{1},T_{2}\right)
_{W_{1},H}^{\prime }}\left( \left( \bigcap\limits_{\left( \varphi ,\psi
\right) \in \left[ \left( \mu ,\nu \right) \right] }\ker \psi \right) \cap
A\left( X_{1}\right) Y_{1}\right) =
\end{equation*}%
\begin{equation*}
\bigcap\limits_{\left( \mu ,\nu \right) \in \left( T_{1},T_{2}\right)
_{W_{1},H}^{\prime }}\ker \nu =T_{2}.
\end{equation*}
\end{proof}

\begin{theorem}
\label{cl}If $\left( L_{1},V_{1}\right) =H_{1},\left( L_{2},V_{2}\right)
=H_{2}\in \Xi $ and $Cl_{H_{1}}\left( W_{2}\right) =Cl_{H_{2}}\left(
W_{2}\right) $, then $Cl_{H_{1}}\left( W_{1}\right) =Cl_{H_{2}}\left(
W_{1}\right) $.
\end{theorem}

\begin{proof}
We consider $\left( T_{1},T_{2}\right) \in Cl_{H_{1}}\left( W_{1}\right) $.
By Proposition \ref{cl2} $\left( T_{1},T_{2}\right) =\left(
T_{1},T_{2}\right) _{W_{2},H_{1}}^{\prime \prime }\cap W_{1}$. $\left(
T_{1},T_{2}\right) _{W_{2},H_{1}}^{\prime \prime }\in Cl_{H_{1}}\left(
W_{2}\right) =Cl_{H_{2}}\left( W_{2}\right) $. Therefore, by Proposition \ref%
{cl1}, $\left( T_{1},T_{2}\right) _{W_{2},H_{1}}^{\prime \prime }\cap
W_{1}=\left( T_{1},T_{2}\right) \in Cl_{H_{2}}\left( W_{1}\right) $.
\end{proof}

\section{Representations of Lie algebras and Lie algebras with
projection-derivation.}

\setcounter{equation}{0}

It is well known that if we have a representation of the Lie algebra $\left(
L,V\right) $ then in the $k$-linear space $M=L\oplus V$ we can define the
structure of Lie algebra if we define the new Lie brackets $\left[ ,\right]
_{M}$ by this formula%
\begin{equation}
\left[ l_{1}+v_{1},l_{2}+v_{2}\right] _{M}=\left[ l_{1},l_{2}\right]
+l_{1}\circ v_{2}-l_{2}\circ v_{1},  \label{new_brackets}
\end{equation}%
where $l_{1},l_{2}\in L$, $v_{1},v_{2}\in V$.

We will denote by $p$ the projection of $M$ on the linear subspace $V$. $%
p\left( l+v\right) =v$ for every $l\in L$, $v\in V$. We have

\begin{equation*}
p\left[ l_{1}+v_{1},l_{2}+v_{2}\right] _{M}=p\left( \left[ l_{1},l_{2}\right]
+l_{1}\circ v_{2}-l_{2}\circ v_{1}\right) =l_{1}\circ v_{2}-l_{2}\circ v_{1},
\end{equation*}%
\begin{equation*}
\left[ p\left( l_{1}+v_{1}\right) ,l_{2}+v_{2}\right] _{M}+\left[
l_{1}+v_{1},p\left( l_{2}+v_{2}\right) \right] _{M}=
\end{equation*}%
\begin{equation*}
\left[ v_{1},l_{2}+v_{2}\right] _{M}+\left[ l_{1}+v_{1},v_{2}\right]
_{M}=-l_{2}\circ v_{1}+l_{1}\circ v_{2}
\end{equation*}%
for every $l_{1},l_{2}\in L$, $v_{1},v_{2}\in V$. Therefore in the new Lie
algebra $p$ will be a derivation. We call these algebras: Lie algebras with
projection-derivation and denote $\left( M,p\right) $.

Vice versa, if we assume that we have a Lie algebra with
projection-derivation $\left( M,p\right) $ then we have the decomposition of
the $k$-linear space $M=\ker p\oplus \mathrm{im}p$. If we denote $\ker p=L$, 
$\mathrm{im}p=V$, then we can prove this proposition:

\begin{proposition}
If we consider $L$ with the Lie brackets inducted from $M$ then $L$ is a Lie
algebra. If we define 
\begin{equation}
l\circ v=\left[ l,v\right]  \label{Maction}
\end{equation}
for every $l\in L$ and every $v\in V$ then $\left( L,V\right) $ is a
representations of the Lie algebra $L$ over the linear space $V$.
\end{proposition}

\begin{proof}
If $l_{1},l_{2}\in \ker p$ then $p\left[ l_{1},l_{2}\right] =\left[ p\left(
l_{1}\right) ,l_{2}\right] +\left[ l_{1},p\left( l_{2}\right) \right] =0$,
so $L=\ker p$ is a Lie algebra.

If $l\in \ker p$, $v\in \mathrm{im}p$ then $p\left[ l,v\right] =\left[
p\left( l\right) ,v\right] +\left[ l,p\left( v\right) \right] =\left[ l,v%
\right] $, so $\left[ l,v\right] =l\circ v\in \mathrm{im}p$.

If $l_{1},l_{2}\in \ker p$, $v\in \mathrm{im}p$ then 
\begin{equation*}
\left[ l_{1},l_{2}\right] \circ v=\left[ \left[ l_{1},l_{2}\right] ,v\right]
=-\left[ \left[ l_{2},v\right] ,l_{1}\right] -\left[ \left[ v,l_{1}\right]
,l_{2}\right] =
\end{equation*}%
\begin{equation*}
\left[ l_{1},\left[ l_{2},v\right] \right] -\left[ l_{2},\left[ l_{1},v%
\right] \right] =l_{1}\circ \left( l_{2}\circ v\right) -l_{2}\circ \left(
l_{1}\circ v\right) .
\end{equation*}%
Also we have for $v_{1},v_{2}\in \mathrm{im}p$ then $p\left[ v_{1},v_{2}%
\right] =\left[ p\left( v_{1}\right) ,v_{2}\right] +\left[ v_{1},p\left(
v_{2}\right) \right] =\left[ v_{1},v_{2}\right] +\left[ v_{1},v_{2}\right] $%
. $char\left( k\right) \neq 2$, so $\left[ v_{1},v_{2}\right] \in \mathrm{im}%
p$, $p\left[ v_{1},v_{2}\right] =\left[ v_{1},v_{2}\right] $ and $\left[
v_{1},v_{2}\right] =0$. Therefore $\left( L,V\right) $ is a representation
of the Lie algebra $L$ over the linear space $V$.
\end{proof}

\begin{proposition}
\label{homomorphismsCorresp}We assume that $\left( \varphi ,\psi \right)
:\left( L_{1},V_{1}\right) \rightarrow \left( L_{2},V_{2}\right) $ is a
homomorphism of representations. Then $f=\varphi \oplus \psi :\left(
L_{1}\oplus V_{1},p_{V_{1}}\right) \rightarrow \left( L_{2}\oplus
V_{2},p_{V_{2}}\right) $, which define by formula $f\left( l+v\right)
=\varphi \left( l\right) +\psi \left( v\right) $ for every $l\in L_{1}$, $%
v\in V_{1}$ is a homomorphism of the Lie algebras with projection-derivation
and $\ker f=\ker \varphi \oplus \ker \psi $. Vice versa, if $f:\left(
M_{1},p_{1}\right) \rightarrow \left( M_{2},p_{2}\right) $ is a homomorphism
of the Lie algebras with projection-derivation then $\left( r_{2}f\kappa
_{1},p_{2}f\iota _{1}\right) :\left( \ker p_{1},\mathrm{im}p_{1}\right)
\rightarrow \left( \ker p_{2},\mathrm{im}p_{2}\right) $, where $%
r_{2}=id_{M_{2}}-p_{2}$ and $\kappa _{1}:\ker p_{1}\hookrightarrow M_{1}$, $%
\iota _{1}:\mathrm{im}p_{1}\hookrightarrow M_{1}$ are embeddings, is a
homomorphism of the representations of the Lie algebras and $\ker
r_{2}f\kappa _{1}=\ker f\cap \ker p_{1}$, $\ker p_{2}f\iota _{1}=\ker f\cap 
\mathrm{im}p_{1}$.
\end{proposition}

\begin{proof}
For the sake of brevity hear and in other proves we denote the various Lie
brackets, projections and embeddings by similar symbols. It should not cause
confusion because we not cause confusion when, for example, in the various
groups denote multiplication, taking the inverse element and unit by similar
symbols.

If $\left( \varphi ,\psi \right) :\left( L_{1},V_{1}\right) \rightarrow
\left( L_{2},V_{2}\right) $ is a homomorphism of representations then $%
f=\varphi \oplus \psi $ is a linear mapping. If $l_{1},l_{2}\in L_{1}$, $%
v_{1},v_{2}\in V_{1}$ then 
\begin{equation*}
f\left[ l_{1}+v_{1},l_{2}+v_{2}\right] =f\left( \left[ l_{1},l_{2}\right]
+l_{1}\circ v_{2}-l_{2}\circ v_{1}\right) =
\end{equation*}%
\begin{equation*}
\varphi \left[ l_{1},l_{2}\right] +\psi \left( l_{1}\circ v_{2}\right) -\psi
\left( l_{2}\circ v_{1}\right) =\left[ \varphi \left( l_{1}\right) ,\varphi
\left( l_{2}\right) \right] +\varphi \left( l_{1}\right) \circ \psi \left(
v_{2}\right) -\varphi \left( l_{2}\right) \circ \psi \left( v_{1}\right) .
\end{equation*}%
\begin{equation*}
\left[ f\left( l_{1}+v_{1}\right) ,f\left( l_{2}+v_{2}\right) \right] =\left[
\varphi \left( l_{1}\right) +\psi \left( v_{1}\right) ,\varphi \left(
l_{2}\right) +\psi \left( v_{2}\right) \right] =
\end{equation*}%
\begin{equation*}
\left[ \varphi \left( l_{1}\right) ,\varphi \left( l_{2}\right) \right]
+\varphi \left( l_{1}\right) \circ \psi \left( v_{2}\right) -\varphi \left(
l_{2}\right) \circ \psi \left( v_{1}\right) .
\end{equation*}

If $l\in L_{1}$, $v\in V_{1}$ then 
\begin{equation*}
fp\left( l+v\right) =f\left( v\right) =\psi \left( v\right) ,
\end{equation*}%
\begin{equation*}
pf\left( l+v\right) =p\left( \varphi \left( l\right) +\psi \left( v\right)
\right) =\psi \left( v\right) .
\end{equation*}%
So $f$ is a homomorphism of the Lie algebras with projection-derivation.

It is clear that $\ker f\supseteq \ker \varphi \oplus \ker \psi $. If $l\in
L_{1}$, $v\in V_{1}$ and $f\left( l+v\right) =\varphi \left( l\right) +\psi
\left( v\right) =0$, then, because $\varphi \left( l\right) \in L_{2}$, $%
\psi \left( v\right) \in V_{2}$, $\varphi \left( l\right) =0$, $\psi \left(
v\right) =0$. So $\ker f=\ker \varphi \oplus \ker \psi $.

Now we assume that $f:\left( M_{1},p_{1}\right) \rightarrow \left(
M_{2},p_{2}\right) $ is a homomorphism of the Lie algebras with
projection-derivation. $pr=p\left( id-p\right) =p-p^{2}=0$, so $rf\kappa
:\ker p\rightarrow \ker p$. Also is clear that $pf\iota :\mathrm{im}%
p\rightarrow \mathrm{im}p$.

It is clear that $rf\kappa $ and $pf\iota $ are linear mappings. For every $%
l\in \ker p$ we have $pf\left( l\right) =fp\left( l\right) =0$. So we have
for every $l_{1},l_{2}\in \ker p$%
\begin{equation*}
rf\kappa \left[ l_{1},l_{2}\right] =r\left[ f\left( l_{1}\right) ,f\left(
l_{2}\right) \right] =\left( id-p\right) \left[ f\left( l_{1}\right)
,f\left( l_{2}\right) \right] =
\end{equation*}%
\begin{equation*}
\left[ f\left( l_{1}\right) ,f\left( l_{2}\right) \right] -\left[ pf\left(
l_{1}\right) ,f\left( l_{2}\right) \right] -\left[ f\left( l_{1}\right)
,pf\left( l_{2}\right) \right] =\left[ f\left( l_{1}\right) ,f\left(
l_{2}\right) \right] .
\end{equation*}%
\begin{equation*}
\left[ rf\kappa \left( l_{1}\right) ,rf\kappa \left( l_{2}\right) \right] =%
\left[ \left( id-p\right) f\left( l_{1}\right) ,\left( id-p\right) f\left(
l_{2}\right) \right] =
\end{equation*}%
\begin{equation*}
\left[ f\left( l_{1}\right) ,f\left( l_{2}\right) \right] -\left[ pf\left(
l_{1}\right) ,f\left( l_{2}\right) \right] -\left[ f\left( l_{1}\right)
,pf\left( l_{2}\right) \right] +\left[ pf\left( l_{1}\right) ,pf\left(
l_{2}\right) \right] =
\end{equation*}%
\begin{equation*}
\left[ f\left( l_{1}\right) ,f\left( l_{2}\right) \right] .
\end{equation*}%
So $rf\kappa $ is a homomorphism of the Lie algebras. If $l\in \ker p$, $%
v\in \mathrm{im}p$, then 
\begin{equation*}
rf\kappa \left( l\right) \circ pf\iota \left( v\right) =\left[ rf\left(
l\right) ,pf\left( v\right) \right] =
\end{equation*}%
\begin{equation*}
\left[ f\left( l\right) ,pf\left( v\right) \right] -\left[ pf\left( l\right)
,pf\left( v\right) \right] =\left[ f\left( l\right) ,pf\left( v\right) %
\right] .
\end{equation*}%
\begin{equation*}
pf\iota \left( l\circ v\right) =pf\left[ l,v\right] =p\left[ f\left(
l\right) ,f\left( v\right) \right] =
\end{equation*}%
\begin{equation*}
\left[ pf\left( l\right) ,f\left( v\right) \right] +\left[ f\left( l\right)
,pf\left( v\right) \right] =\left[ f\left( l\right) ,pf\left( v\right) %
\right] .
\end{equation*}%
So $\left( rf\kappa ,pf\iota \right) $ is a homomorphism of the
representations of the Lie algebras.

It is clear that $\ker f\cap \ker p\subseteq \ker rf\kappa $. If $l\in \ker
p $ and $rf\kappa \left( l\right) =0$ then $l=r\left( l\right) $ and $%
f\left( l\right) =fr\left( l\right) =rf\kappa \left( l\right) =0$. So $\ker
rf\kappa =\ker f\cap \ker p$. Analogously $\ker pf\iota =\ker f\cap \mathrm{%
im}p$.
\end{proof}

We denote by $\Theta $ the variety of all Lie algebras with
projection-derivation. The elements of this variety are Lie algebras with
all operations and axioms of Lie algebras and with one additional unary
operation: projection $p$, which fulfills two axioms of linear map and two
additional axioms:

\begin{enumerate}
\item $p\left( p\left( m\right) \right) =p\left( m\right) $ holds for every $%
m\in M$,

\item $p\left[ m_{1},m_{2}\right] =\left[ p\left( m_{1}\right) ,m_{2}\right]
+\left[ m_{1},p\left( m_{2}\right) \right] $ holds for every $m_{1},m_{2}\in
M$,
\end{enumerate}

where $M\in \Theta $.

We can consider the varieties $\Xi $ and $\Theta $ as categories. The
objects of these categories are universal algebras from these varieties and
morphisms are homomorphisms. We have a functor $\mathcal{F}:\Xi \rightarrow
\Theta $, such that 
\begin{equation*}
\mathcal{F}\left( L,V\right) =\left( L\oplus V,p_{V}\right)
\end{equation*}%
for $\left( L,V\right) \in \mathrm{Ob}\Xi $, 
\begin{equation*}
\mathcal{F}\left( \left( \varphi ,\psi \right) :\left( L_{1},V_{1}\right)
\rightarrow \left( L_{2},V_{2}\right) \right) =\varphi \oplus \psi :\left(
L_{1}\oplus V_{1},p_{V_{1}}\right) \rightarrow \left( L_{2}\oplus
V_{2},p_{V_{2}}\right)
\end{equation*}%
for $\left( \varphi ,\psi \right) \in \mathrm{Mor}\Xi $.

Also we have a functor $\mathcal{F}^{-1}:\Theta \rightarrow \Xi $, such that 
\begin{equation*}
\mathcal{F}^{-1}\left( M,p\right) =\left( \ker p,\mathrm{im}p\right)
\end{equation*}%
for $\left( M,p\right) \in \mathrm{Ob}\Theta $, 
\begin{equation*}
\mathcal{F}^{-1}\left( f:\left( M_{1},p_{1}\right) \rightarrow \left(
M_{2},p_{2}\right) \right) =\left( rf\kappa ,pf\iota \right) :\left( \ker
p_{1},\mathrm{im}p_{1}\right) \rightarrow \left( \ker p_{2},\mathrm{im}%
p_{2}\right)
\end{equation*}%
for $f\in \mathrm{Mor}\Theta $.

It is easy to check that $\mathcal{FF}^{-1}=id_{\Theta }$, $\mathcal{F}^{-1}%
\mathcal{F}=id_{\Xi }$ so these functors are isomorphisms of categories.

\begin{theorem}
\label{freeWithProj}If $\left( F,p\right) =F\left( m_{1},\ldots
,m_{n}\right) $ is a free Lie algebras with projection-derivation with free
generators $\left\{ m_{1},\ldots ,m_{n}\right\} $ then $\mathcal{F}%
^{-1}\left( F,p\right) =\left( L,V\right) $ is a free representation with
the pair of sets of the free generators $\left\{ X,Y\right\} $, where $%
X=\left\{ r\left( m_{1}\right) ,\ldots ,r\left( m_{n}\right) \right\} $ and $%
Y=\left\{ p\left( m_{1}\right) ,\ldots ,p\left( m_{n}\right) \right\} $.
\end{theorem}

\begin{proof}
It is clear that $X\subset \ker p$, $Y\subset \mathrm{im}p$. We will
consider an arbitrary $\left( Q,U\right) \in \Xi $. We assume that we have $%
2 $ mappings: 
\begin{equation*}
\varphi :\left\{ r\left( m_{1}\right) ,\ldots ,r\left( m_{n}\right) \right\}
\ni r\left( m_{i}\right) \rightarrow q_{i}\in Q
\end{equation*}%
and 
\begin{equation*}
\psi :\left\{ p\left( m_{1}\right) ,\ldots ,p\left( m_{n}\right) \right\}
\ni p\left( m_{i}\right) \rightarrow u_{i}\in U.
\end{equation*}%
So we have a mapping 
\begin{equation*}
f:\left\{ m_{1},\ldots ,m_{n}\right\} \ni m_{i}\rightarrow
q_{i}+u_{i}=\varphi r\left( m_{i}\right) +\psi p\left( m_{i}\right) \in
Q\oplus U.
\end{equation*}%
Hence, by our assumption about $\left( F,p\right) $, this mapping can be
extended to the homomorphism 
\begin{equation*}
f:\left( F,p\right) \rightarrow \mathcal{F}\left( Q,U\right) =\left( Q\oplus
U,p_{U}\right) .
\end{equation*}%
So there is a homomorphism 
\begin{equation*}
\mathcal{F}^{-1}\left( f\right) =\left( rf\kappa ,pf\iota \right) :\mathcal{F%
}^{-1}\left( F,p\right) =\left( \ker p,\mathrm{im}p\right) \rightarrow
\left( Q,U\right) .
\end{equation*}%
\begin{equation*}
rf\kappa \left( r\left( m_{i}\right) \right) =\left( r\right) ^{2}f\left(
m_{i}\right) =r\left( \varphi r\left( m_{i}\right) +\psi p\left(
m_{i}\right) \right) =\varphi r\left( m_{i}\right) ,
\end{equation*}%
\begin{equation*}
pf\iota \left( p\left( m_{i}\right) \right) =p\left( \varphi r\left(
m_{i}\right) +\psi p\left( m_{i}\right) \right) =\psi p\left( m_{i}\right)
\end{equation*}%
holds for $1\leq i\leq n$, because $\varphi r\left( m_{i}\right) \in Q$, $%
\psi p\left( m_{i}\right) \in U$.
\end{proof}

We will denote 
\begin{equation*}
\Xi ^{\prime }=\left\{ W\left( X,Y\right) \in \mathrm{Ob}\Xi ^{0}\mid
\left\vert X\right\vert =\left\vert Y\right\vert \right\} .
\end{equation*}

\begin{theorem}
\label{revFreeWithProj}If $W=\left( L,V\right) =\left( L\left( x_{1},\ldots
,x_{n}\right) ,\bigoplus\limits_{i=1}^{n}A\left( x_{1},\ldots ,x_{n}\right)
y_{i}\right) \in \Xi ^{\prime }$ then $\mathcal{F}\left( W\right) =\left(
F,p\right) $ is a free Lie algebra with projection-derivation which has $n$
free generators $m_{i}=x_{i}+y_{i}$, $1\leq i\leq n$.
\end{theorem}

\begin{proof}
For $1\leq i\leq n$ we have that $x_{i}\in L$, $y_{i}\in V$, so $%
m_{i}=x_{i}+y_{i}\in F=L\oplus V$. We will consider an arbitrary $\left(
N,p\right) \in \Theta $. We assume that we have a mapping%
\begin{equation*}
f:\left\{ m_{1},\ldots ,m_{n}\right\} \ni m_{i}\rightarrow n_{i}\in N.
\end{equation*}%
We will construct two other mappings%
\begin{equation*}
\varphi :\left\{ x_{1},\ldots ,x_{n}\right\} \ni x_{i}\rightarrow r\left(
n_{i}\right) \in \ker p\subset N
\end{equation*}%
and%
\begin{equation*}
\psi :\left\{ y_{1},\ldots ,y_{n}\right\} \ni y_{i}\rightarrow p\left(
n_{i}\right) \in \mathrm{im}p\subset N.
\end{equation*}%
By our assumption about $\left( L,V\right) $, these mappings can be extended
to the homomorphism%
\begin{equation*}
\left( \varphi ,\psi \right) :\left( L,V\right) \rightarrow \mathcal{F}%
^{-1}\left( N,p\right) =\left( \ker p,\mathrm{im}p\right) .
\end{equation*}%
So there is a homomorphism 
\begin{equation*}
\mathcal{F}\left( \varphi ,\psi \right) =\left( \varphi \oplus \psi \right) :%
\mathcal{F}\left( L,V\right) =\left( F,p\right) \rightarrow \left(
N,p\right) .
\end{equation*}%
\begin{equation*}
\left( \varphi \oplus \psi \right) \left( m_{i}\right) =\left( \varphi
\oplus \psi \right) \left( x_{i}+y_{i}\right) =\varphi \left( x_{i}\right)
+\psi \left( y_{i}\right) =
\end{equation*}%
\begin{equation*}
r\left( n_{i}\right) +p\left( n_{i}\right) =n_{i}=f\left( m_{i}\right)
\end{equation*}%
holds for $1\leq i\leq n$, because $x_{i}\in L$, $y_{i}\in V$.
\end{proof}

\section{Automorphisms of the category $\Xi ^{0}$ and of the category $%
\Theta ^{0}$.\label{From2Autom_to_1Autom}}

\setcounter{equation}{0}

If we have a category $\mathfrak{K}$, which objects are universal algebras
and morphisms are homomorphism, then automorphism $\Phi $ of this category
transform the homomorphism $id_{A}\in \mathrm{Mor}\mathfrak{K}$, where $A\in 
\mathrm{Ob}\mathfrak{K}$, to homomorphism $id_{\Phi \left( A\right) }$,
because homomorphism $id_{A}$ uniquely defined by its "algebraic" property: $%
id_{A}$ is a unit of the semigroup $\mathrm{End}A$. Therefore if $A,B\in 
\mathrm{Ob}\mathfrak{K}$, $A\cong B$, $\Phi \in \mathrm{Aut}\mathfrak{K}$
then $\Phi \left( A\right) \cong \Phi \left( B\right) $.

\begin{theorem}
The category $\Xi ^{0}$ has $2$-sorted IBN propriety: if $W\left(
X_{1},Y_{1}\right) $, $W\left( X_{2},Y_{2}\right) \in \mathrm{Ob}\Xi ^{0}$
and $W\left( X_{1},Y_{1}\right) \cong W\left( X_{2},Y_{2}\right) $, then $%
\left\vert X_{1}\right\vert =\left\vert X_{2}\right\vert $ and $\left\vert
Y_{1}\right\vert =\left\vert Y_{2}\right\vert $.
\end{theorem}

\begin{proof}
We consider $W\left( X,Y\right) =\left( L\left( X\right) ,A\left( X\right)
Y\right) =\left( L,V\right) \in \mathrm{Ob}\Xi ^{0}$. $L/L^{2}$ is a $k$%
-linear space and $\dim L/L^{2}=\left\vert X\right\vert $.

In the associative algebra $A\left( X\right) $ we will consider $%
\left\langle L\right\rangle $ - two-sided ideal generated by the set $%
L=L\left( X\right) \subset A\left( X\right) $. This ideal coincide with $%
\left\langle X\right\rangle $ - two-sided ideal generated by the set $X$,
because $L$ is a subset of the associative subalgebra without unit,
generated by the set $X$. In the $A\left( X\right) $-module $V=A\left(
X\right) Y$ we will consider submodule $\left\langle L\right\rangle V=%
\mathrm{Span}_{k}\left( av\mid a\in \left\langle L\right\rangle ,v\in
V\right) =\left\langle X\right\rangle V$. $\dim V/\left\langle
X\right\rangle V=\dim V/\left\langle L\right\rangle V=\left\vert
Y\right\vert $.

We assume that we have two objects of $\Xi ^{0}$: $W\left(
X_{1},Y_{1}\right) =\left( L\left( X_{1}\right) ,A\left( X_{1}\right)
Y_{1}\right) =\left( L_{1},V_{1}\right) $ and $W\left( X_{2},Y_{2}\right)
=\left( L\left( X_{2}\right) ,A\left( X_{2}\right) Y_{2}\right) =\left(
L_{2},V_{2}\right) $ - and there is an isomorphism $\left( \varphi ,\psi
\right) :\left( L_{1},V_{1}\right) \rightarrow \left( L_{2},V_{2}\right) $.
It means that $\varphi :$ $L_{1}\rightarrow L_{2}$ is an isomorphism and $%
L_{1}/L_{1}^{2}\cong L_{2}/L_{2}^{2}$, so $\left\vert X_{1}\right\vert
=\left\vert X_{2}\right\vert $.

By (\ref{homomorphism_formula}) we have that $\psi :V_{1}\rightarrow V_{2}$
is an isomorphism of the $L_{1}$-modules, when acting of $L_{1}$ over $V_{2}$
defined by $l\circ v=\varphi \left( l\right) v$, where $l\in L_{1}$, $v\in
V_{2}$. $A\left( X_{1}\right) $ and $A\left( X_{2}\right) $ are universal
enveloped algebras of $L_{1}$ and $L_{2}$ respectively. Therefore the
isomorphism $\varphi :$ $L_{1}\rightarrow L_{2}$ can be extended to the
isomorphism of algebras with unit $\varphi :A\left( X_{1}\right) \rightarrow
A\left( X_{2}\right) $. $A\left( X_{1}\right) $ is generated as algebra with
unit by elements of $L_{1}$, so $\psi $ is also an isomorphism of the $%
A\left( X_{1}\right) $-modules. Therefore there is an isomorphism of $%
A\left( X_{1}\right) $-modules $V_{1}/\left\langle L_{1}\right\rangle
V_{1}\cong V_{2}/\left\langle L_{2}\right\rangle V_{2}$, because $\psi
\left( \left\langle L_{1}\right\rangle V_{1}\right) =\left\langle
L_{2}\right\rangle V_{2}$. So $\dim V_{1}/\left\langle L_{1}\right\rangle
V_{1}=\dim V_{2}/\left\langle L_{2}\right\rangle V_{2}$ and $\left\vert
Y_{1}\right\vert =\left\vert Y_{2}\right\vert $.
\end{proof}

By this Theorem we have that if $W\left( X_{1},Y_{1}\right) ,W\left(
X_{2},Y_{2}\right) \in \mathrm{Ob}\Xi ^{0}$, $\left\vert X_{1}\right\vert
=\left\vert X_{2}\right\vert $, $\left\vert Y_{1}\right\vert =\left\vert
Y_{2}\right\vert $, $\Phi \in \mathrm{Aut}\Xi ^{0}$, $\Phi \left( W\left(
X_{1},Y_{1}\right) \right) =W\left( X_{3},Y_{3}\right) $, $\Phi \left(
W\left( X_{2},Y_{2}\right) \right) =W\left( X_{4},Y_{4}\right) $ then $%
\left\vert X_{3}\right\vert =\left\vert X_{4}\right\vert $, $\left\vert
Y_{3}\right\vert =\left\vert Y_{4}\right\vert $.

This is a well-known

\begin{definition}
We consider a category $\mathfrak{K}$ and the family of objects $\left\{
A_{i}\right\} _{i\in I}\subseteq \mathrm{Ob}\mathfrak{K}$. The pair $\left(
Q\in \mathrm{Ob}\mathfrak{K,}\left\{ \eta _{i}:A_{i}\rightarrow Q\right\}
_{i\in I}\subseteq \mathrm{Mor}\mathfrak{K}\right) $\ called \textbf{%
coproduct} of $\left\{ A_{i}\right\} _{i\in I}$ if for every $B\in \mathrm{Ob%
}\mathfrak{K}$ and every $\left\{ \alpha _{i}:A_{i}\rightarrow B\right\}
_{i\in I}\subseteq \mathrm{Mor}\mathfrak{K}$ there exists only one $\alpha
:Q\rightarrow B\in \mathrm{Mor}\mathfrak{K}$ such that $\alpha _{i}=\alpha
\eta _{i}$.
\end{definition}

We denote $Q=\coprod\limits_{i\in I}A_{i}$. It is clear that if $%
Q=\coprod\limits_{i\in I}A_{i}$, $\Phi \in \mathrm{Aut}\mathfrak{K}$ then $%
\Phi \left( Q\right) =\coprod\limits_{i\in I}\Phi \left( A_{i}\right) $.

It is easy to check that if $W\left( X_{1},Y_{1}\right) ,W\left(
X_{2},Y_{2}\right) \in \mathrm{Ob}\Xi ^{0}$ then $W\left( X_{1},Y_{1}\right)
\sqcup W\left( X_{2},Y_{2}\right) =W\left( X_{3},Y_{3}\right) $, where $%
\left\vert X_{3}\right\vert =\left\vert X_{1}\right\vert +\left\vert
X_{2}\right\vert $, $\left\vert Y_{3}\right\vert =\left\vert
Y_{1}\right\vert +\left\vert Y_{2}\right\vert $. Of course, homomorphisms $%
\eta _{i}:W\left( X_{i},Y_{i}\right) \rightarrow W\left( X_{3},Y_{3}\right) $%
, $i=1,2$, must by constructed by suitable way.

Similar to \cite{PlotkinVovsi} we define

\begin{definition}
We say that the free representation $W\left( X,Y\right) $ is a \textbf{cyclic%
} if $X=\left\{ x\right\} $, $Y=\left\{ y\right\} $.
\end{definition}

\begin{proposition}
\label{cyclic_invariant}If $\left( L,V\right) \in \mathrm{Ob}\Xi ^{0}$ is a
cyclic representation, $\Phi \in \mathrm{Aut}\Xi ^{0}$ then $\Phi \left(
L,V\right) $ is also a cyclic representation.
\end{proposition}

\begin{proof}
We consider arbitrary $\Phi \in \mathrm{Aut}\Xi ^{0}$. We take $x\in X^{0}$, 
$y\in Y^{0}$. $\Phi \left( kx,\left\{ 0\right\} \right) =W\left(
X_{1},Y_{1}\right) $, $\Phi \left( \left\{ 0\right\} ,ky\right) =W\left(
X_{2},Y_{2}\right) $. Also there exist $W\left( X_{3},Y_{3}\right) $, $%
W\left( X_{4},Y_{4}\right) \in \mathrm{Ob}\Xi ^{0}$, such that $\Phi \left(
W\left( X_{3},Y_{3}\right) \right) =\left( kx,\left\{ 0\right\} \right) $, $%
\Phi \left( W\left( X_{4},Y_{4}\right) \right) =\left( \left\{ 0\right\}
,ky\right) $. We denote $\left\vert X_{i}\right\vert =a_{i}$, $\left\vert
Y_{i}\right\vert =b_{i}$, $i=1,\ldots ,4$. We have that $W\left(
X_{3},\varnothing \right) =\coprod\limits_{i=1}^{a_{3}}\left( kx_{i},\left\{
0\right\} \right) $, $W\left( \varnothing ,Y_{3}\right)
=\coprod\limits_{i=1}^{b_{3}}\left( \left\{ 0\right\} ,ky_{i}\right) $,
where $x_{i}\in X^{0}$, $y_{i}\in Y^{0}$. 
\begin{equation*}
W\left( X_{3},Y_{3}\right) =W\left( X_{3},\varnothing \right) \sqcup W\left(
\varnothing ,Y_{3}\right) =\left( \coprod\limits_{i=1}^{a_{3}}\left(
kx_{i},\left\{ 0\right\} \right) \right) \sqcup \left(
\coprod\limits_{i=1}^{b_{3}}\left( \left\{ 0\right\} ,ky_{i}\right) \right) ,
\end{equation*}%
\begin{equation*}
\Phi \left( W\left( X_{3},Y_{3}\right) \right) =\left(
\coprod\limits_{i=1}^{a_{3}}\Phi \left( kx_{i},\left\{ 0\right\} \right)
\right) \sqcup \left( \coprod\limits_{i=1}^{b_{3}}\Phi \left( \left\{
0\right\} ,ky_{i}\right) \right) .
\end{equation*}%
$\Phi \left( kx_{i},\left\{ 0\right\} \right) \cong W\left(
X_{1},Y_{1}\right) $, $\Phi \left( \left\{ 0\right\} ,ky_{i}\right) \cong
W\left( X_{2},Y_{2}\right) $, so $\coprod\limits_{i=1}^{a_{3}}\Phi \left(
kx_{i},\left\{ 0\right\} \right) =W\left( X_{5},Y_{5}\right) $, such that $%
\left\vert X_{5}\right\vert =a_{3}a_{1}$, $\left\vert Y_{5}\right\vert
=a_{3}b_{1}$ and $\coprod\limits_{i=1}^{b_{3}}\Phi \left( \left\{ 0\right\}
,ky_{i}\right) =W\left( X_{6},Y_{6}\right) $, such that $\left\vert
X_{6}\right\vert =b_{3}a_{2}$, $\left\vert Y_{6}\right\vert =b_{3}b_{2}$.
Therefore $\Phi \left( W\left( X_{3},Y_{3}\right) \right) =W\left(
X_{7},Y_{7}\right) $, where $\left\vert X_{7}\right\vert
=a_{3}a_{1}+b_{3}a_{2}$, $\left\vert Y_{7}\right\vert =a_{3}b_{1}+b_{3}b_{2}$%
. But $\Phi \left( W\left( X_{3},Y_{3}\right) \right) =\left( kx,\left\{
0\right\} \right) $, therefore%
\begin{equation}
a_{3}a_{1}+b_{3}a_{2}=1  \label{eqv1}
\end{equation}%
\begin{equation}
a_{3}b_{1}+b_{3}b_{2}=0.  \label{eqv2}
\end{equation}%
By consideration of $\Phi \left( W\left( X_{4},Y_{4}\right) \right) $ we
conclude by similar way, that%
\begin{equation}
a_{4}a_{1}+b_{4}a_{2}=0  \label{eqv3}
\end{equation}%
\begin{equation}
a_{4}b_{1}+b_{4}b_{2}=1  \label{eqv4}
\end{equation}%
$a_{i},b_{i}\in 
\mathbb{N}
$, so from (\ref{eqv2}) we conclude $a_{3}b_{1}=0$ and $b_{3}b_{2}=0$. If $%
a_{3}=0$, then by (\ref{eqv1}) we have $b_{3}a_{2}=1$ and $b_{3}=1$, $%
a_{2}=1 $. So $b_{2}=0$ and from (\ref{eqv4}) $a_{4}=1$, $b_{1}=1$. After
this we conclude from (\ref{eqv3}) that $a_{1}=0$, $b_{4}=0$. In this case $%
\Phi \left( kx,\left\{ 0\right\} \right) =\left( \left\{ 0\right\}
,ky^{\prime }\right) $, $\Phi \left( \left\{ 0\right\} ,ky\right) =\left(
kx^{\prime },\left\{ 0\right\} \right) $, where $x^{\prime }\in X^{0}$, $%
y^{\prime }\in Y^{0}$. So%
\begin{equation*}
\Phi \left( kx,ky\right) =\Phi \left( \left( kx,\left\{ 0\right\} \right)
\sqcup \left( \left\{ 0\right\} ,ky\right) \right) =\Phi \left( kx,\left\{
0\right\} \right) \sqcup \Phi \left( \left\{ 0\right\} ,ky\right) =
\end{equation*}%
\begin{equation*}
\left( \left\{ 0\right\} ,ky^{\prime }\right) \sqcup \left( kx^{\prime
},\left\{ 0\right\} \right) =\left( kx^{\prime },k\left[ x^{\prime }\right]
y^{\prime }\right) .
\end{equation*}

Now we will resolve the system (\ref{eqv1}) - (\ref{eqv4}) in the case $%
b_{1}=0$. In this case we conclude from (\ref{eqv4}) that $b_{4}=1$, $%
b_{2}=1 $. From (\ref{eqv3}) we conclude that $a_{2}=0$, from (\ref{eqv2}) - 
$b_{3}=0 $. After this we conclude from (\ref{eqv1}) that $a_{3}=1$, $%
a_{1}=1 $. So by (\ref{eqv3}) we have $a_{4}=0$. In this case $\Phi \left(
kx,\left\{ 0\right\} \right) =\left( kx^{\prime \prime },\left\{ 0\right\}
\right) $, $\Phi \left( \left\{ 0\right\} ,ky\right) =\left( \left\{
0\right\} ,ky^{\prime \prime }\right) $, where $x^{\prime \prime }\in X^{0}$%
, $y^{\prime \prime }\in Y^{0}$. So%
\begin{equation*}
\Phi \left( kx,ky\right) =\Phi \left( \left( kx,\left\{ 0\right\} \right)
\sqcup \left( \left\{ 0\right\} ,ky\right) \right) =\Phi \left( kx,\left\{
0\right\} \right) \sqcup \Phi \left( \left\{ 0\right\} ,ky\right) =
\end{equation*}%
\begin{equation*}
\left( kx^{\prime \prime },\left\{ 0\right\} \right) \sqcup \left( \left\{
0\right\} ,ky^{\prime \prime }\right) =\left( kx^{\prime \prime },k\left[
x^{\prime \prime }\right] y^{\prime \prime }\right) .
\end{equation*}
\end{proof}

\begin{corollary}
\label{X_Y_equal_invariant}If $\Phi \in \mathrm{Aut}\Xi ^{0}$, then $\Phi
\left( \Xi ^{\prime }\right) =\Xi ^{\prime }$.
\end{corollary}

\begin{proof}
\begin{equation*}
\left( L\left( x_{1},\ldots ,x_{n}\right) ,\bigoplus\limits_{i=1}^{n}A\left(
x_{1},\ldots ,x_{n}\right) y_{i}\right) =\coprod\limits_{i=1}^{n}\left(
kx_{i},k\left[ x_{i}\right] y_{i}\right) ,
\end{equation*}%
where $\left\{ x_{1},\ldots ,x_{n}\right\} \subset X_{0}$, $\left\{
y_{1},\ldots ,y_{n}\right\} \subset Y_{0}$, fulfills, so if $W\left(
X,Y\right) \in \mathrm{Ob}\Xi ^{0}$ and $\left\vert X\right\vert =\left\vert
Y\right\vert =n>0$ then $\Phi \left( W\left( X,Y\right) \right) =W\left(
X^{\prime },Y^{\prime }\right) $ and $\left\vert X^{\prime }\right\vert
=\left\vert Y^{\prime }\right\vert =n$.

Now we need only consider the case $X=Y=\varnothing $. We assume that $\Phi
\left( \left\{ 0\right\} ,\left\{ 0\right\} \right) =W\left( X,Y\right) $,
where $\left\vert X\right\vert =a$, $\left\vert Y\right\vert =b$. 
\begin{equation*}
\left( \left\{ 0\right\} ,\left\{ 0\right\} \right) \sqcup \left( \left\{
0\right\} ,\left\{ 0\right\} \right) =\left( \left\{ 0\right\} ,\left\{
0\right\} \right)
\end{equation*}%
so%
\begin{equation*}
\Phi \left( \left\{ 0\right\} ,\left\{ 0\right\} \right) =\Phi \left(
\left\{ 0\right\} ,\left\{ 0\right\} \right) \sqcup \Phi \left( \left\{
0\right\} ,\left\{ 0\right\} \right) =W\left( X,Y\right) \sqcup W\left(
X,Y\right) =W\left( X_{1},Y_{1}\right) ,
\end{equation*}%
where $\left\vert X_{1}\right\vert =2a$, $\left\vert Y_{1}\right\vert =2b$.
So $a=2a$ and $b=2b$, hence $a=b=0$ and $\Phi \left( \left\{ 0\right\}
,\left\{ 0\right\} \right) =\left( \left\{ 0\right\} ,\left\{ 0\right\}
\right) $.
\end{proof}

\setcounter{corollary}{0}

In the category $\Theta $ we can consider subcategory $\Theta ^{0}$. We take
a infinite countable sets of symbols $M^{0}$. The objects of $\Theta ^{0}$
will be the free algebras in $\Theta $ with the set of free generators $M$
such that $M\subset M^{0}$, $\left\vert M\right\vert <\infty $. We will
denote these algebras by $F\left( M\right) $. The morphisms of $\Theta ^{0}$
will be the homomorphisms of these algebras.

By using of the Theorems \ref{freeWithProj} and \ref{revFreeWithProj} $%
\mathcal{F}\left( \Xi ^{\prime }\right) =\mathrm{Ob}\Theta ^{0}$ and $%
\mathcal{F}^{-1}\left( \mathrm{Ob}\Theta ^{0}\right) =\Xi ^{\prime }$, so,
by Corollary \ref{X_Y_equal_invariant} from the Proposition \ref%
{cyclic_invariant} we prove the

\begin{theorem}
\label{withProjAutom}If $\Phi \in \mathrm{Aut}\Xi ^{0}$ then $\mathcal{F}%
\Phi _{\mid \Xi ^{\prime }}\mathcal{F}^{-1}\in \mathrm{Aut}\Theta ^{0}$.
\end{theorem}

\section{Automorphic equivalence in the variety $\Xi $ and in the variety $%
\Theta $.}

\setcounter{equation}{0}

\begin{proposition}
\label{ClCorresp}If $\left( T_{1},T_{2}\right) \subset W=W\left( X,Y\right)
\in \Xi ^{\prime }$, $H=\left( L,V\right) \in \Xi $, $\left(
T_{1},T_{2}\right) $ is an $H$-closed congruence, then $T_{1}\oplus
T_{2}\subset \mathcal{F}\left( W\left( X,Y\right) \right) $ is an $\mathcal{F%
}\left( H\right) $-closed congruence. If $T\subset F\left( M\right) \in 
\mathrm{Ob}\Theta ^{0}$, $N=\left( N,p\right) \in \Theta $, $T$ is an $N$%
-closed congruence, then $\left( T\cap \ker p,T\cap \mathrm{im}p\right)
\subset \mathcal{F}^{-1}\left( F\left( M\right) \right) $ is an $\mathcal{F}%
^{-1}\left( N\right) $-closed congruence. The mappings 
\begin{equation*}
\mathcal{F}_{W,H}:Cl_{H}\left( W\right) \ni \left( T_{1},T_{2}\right)
\rightarrow T_{1}\oplus T_{2}\in Cl_{\mathcal{F}\left( H\right) }\left( 
\mathcal{F}\left( W\right) \right)
\end{equation*}%
and 
\begin{equation*}
\mathcal{F}_{F\left( M\right) ,N}^{-1}:Cl_{N}\left( F\left( M\right) \right)
\ni T\rightarrow \left( T\cap \ker p,T\cap \mathrm{im}p\right) \in Cl_{%
\mathcal{F}^{-1}\left( N\right) }\left( \mathcal{F}^{-1}\left( F\left(
M\right) \right) \right)
\end{equation*}%
are bijections.
\end{proposition}

\begin{proof}
If $\left( \varphi ,\psi \right) \in \left( T_{1},T_{2}\right) _{H}^{\prime
} $, then by Proposition \ref{homomorphismsCorresp} 
\begin{equation*}
\ker \mathcal{F}\left( \varphi ,\psi \right) =\ker \left( \varphi \oplus
\psi \right) =\ker \varphi \oplus \ker \psi \supseteq T_{1}\oplus T_{2},
\end{equation*}%
so 
\begin{equation*}
\mathcal{F}\left( \left( T_{1},T_{2}\right) _{H}^{\prime }\right) =\left\{
f=\varphi \oplus \psi \mid \left( \varphi ,\psi \right) \in \left(
T_{1},T_{2}\right) _{H}^{\prime }\right\} \subseteq \left( T_{1}\oplus
T_{2}\right) _{\mathcal{F}\left( H\right) }^{\prime }.
\end{equation*}%
We will consider $l+v\in \left( T_{1}\oplus T_{2}\right) _{\mathcal{F}\left(
H\right) }^{\prime \prime }=\bigcap\limits_{f\in \left( T_{1}\oplus
T_{2}\right) _{\mathcal{F}\left( H\right) }^{\prime }}\ker f$, where $l\in
\ker p$, $v\in \mathrm{im}p$. $l+v\in \bigcap\limits_{\left( \varphi ,\psi
\right) \in \left( T_{1},T_{2}\right) _{H}^{\prime }}\ker \left( \varphi
\oplus \psi \right) $ holds, so for every $\left( \varphi ,\psi \right) \in
\left( T_{1},T_{2}\right) _{H}^{\prime }$ we have $\left( \varphi \oplus
\psi \right) \left( l+v\right) =\varphi \left( l\right) +\psi \left(
v\right) =0$. $\varphi \left( l\right) \in \ker p$, $\psi \left( v\right)
\in \mathrm{im}p$, so $\varphi \left( l\right) =0$, $\psi \left( v\right) =0$%
. Hence $l\in \bigcap\limits_{\left( \varphi ,\psi \right) \in \left(
T_{1},T_{2}\right) _{H}^{\prime }}\ker \varphi =T_{1}$, $v\in
\bigcap\limits_{\left( \varphi ,\psi \right) \in \left( T_{1},T_{2}\right)
_{H}^{\prime }}\ker \psi =T_{2}$. Therefore $l+v\in T_{1}\oplus T_{2}$ and $%
\left( T_{1}\oplus T_{2}\right) _{\mathcal{F}\left( H\right) }^{\prime
\prime }\subseteq T_{1}\oplus T_{2}$. It means that $T_{1}\oplus T_{2}$ is
an $\mathcal{F}\left( H\right) $-closed congruence.

If $f\in T_{N}^{\prime }$ then $\mathcal{F}^{-1}\left( f\right) =\left(
rf\kappa ,pf\iota \right) $ and by Proposition \ref{homomorphismsCorresp} $%
\ker rf\kappa =\ker f\cap \ker p\supseteq T\cap \ker p$, $\ker pf\iota =\ker
f\cap \mathrm{im}p\supseteq T\cap \mathrm{im}p$ holds. So $\left( rf\kappa
,pf\iota \right) \in \left( T\cap \ker p,T\cap \mathrm{im}p\right) _{%
\mathcal{F}^{-1}\left( N\right) }^{\prime }$ and $\mathcal{F}^{-1}\left(
T_{N}^{\prime }\right) \subseteq \left( T\cap \ker p,T\cap \mathrm{im}%
p\right) _{\mathcal{F}^{-1}\left( N\right) }^{\prime }$. Therefore 
\begin{equation*}
\left( T\cap \ker p,T\cap \mathrm{im}p\right) _{\mathcal{F}^{-1}\left(
N\right) }^{\prime \prime }\subseteq \left( \bigcap\limits_{f\in
T_{N}^{\prime }}\ker rf\kappa ,\bigcap\limits_{f\in T_{N}^{\prime }}\ker
pf\iota \right) =
\end{equation*}%
\begin{equation*}
\left( \bigcap\limits_{f\in T_{N}^{\prime }}\left( \ker f\cap \ker p\right)
,\bigcap\limits_{f\in T_{N}^{\prime }}\left( \ker f\cap \mathrm{im}p\right)
\right) =
\end{equation*}%
\begin{equation*}
\left( \left( \bigcap\limits_{f\in T_{N}^{\prime }}\ker f\right) \cap \ker
p,\left( \bigcap\limits_{f\in T_{N}^{\prime }}\ker f\right) \cap \mathrm{im}%
p\right) =\left( T\cap \ker p,T\cap \mathrm{im}p\right) ,
\end{equation*}%
so $\left( T\cap \ker p,T\cap \mathrm{im}p\right) $ is an $\mathcal{F}%
^{-1}\left( N\right) $-closed congruence.

If $\left( T_{1},T_{2}\right) \in Cl_{H}\left( W\right) $ and $T=T_{1}\oplus
T_{2}\subset \mathcal{F}\left( W\right) $, then $T_{1}=T\cap \ker p$, $%
T_{2}=T\cap \mathrm{im}p$, so $\mathcal{F}_{\mathcal{F}\left( W\right) ,%
\mathcal{F}\left( H\right) }^{-1}\mathcal{F}_{W,H}=id_{Cl_{H}\left( W\right)
}$.

If $N_{1}=\left( N_{1},p_{1}\right) ,N_{2}=\left( N_{2},p_{2}\right) \in
\Theta $ and $f:N_{1}\rightarrow N_{2}$ is a homomorphism then $\ker f$ is $%
p $-invariant. So, $T\in Cl_{N}\left( F\left( M\right) \right) $ is also $p$%
-invariant and $T=\left( T\cap \ker p\right) \oplus \left( T\cap \mathrm{im}%
p\right) $. Therefore $\mathcal{F}_{\mathcal{F}^{-1}\left( F\left( M\right)
\right) ,\mathcal{F}^{-1}\left( N\right) }\mathcal{F}_{F\left( M\right)
,N}^{-1}=id_{Cl_{N}\left( F\left( M\right) \right) }$.
\end{proof}

If $F_{1}=F\left( M_{1}\right) ,F_{2}=F\left( M_{2}\right) \in \mathrm{Ob}%
\Theta ^{0}$ and $T$ is a congruence in $F_{2}$ then $\beta
_{F_{1},F_{2}}\left( T\right) $ will be a relation in $\mathrm{Hom}\left(
F_{1},F_{2}\right) $ which we define as in \cite[Subsection 3.3]{PlotkinSame}%
: $\left( f_{1},f_{2}\right) \in \beta _{F_{1},F_{2}}\left( T\right) $ if
and only if $f_{1}\left( n\right) \equiv f_{2}\left( n\right) \left( \func{%
mod}T\right) $ holds for every $n\in F_{1}$.

\begin{proposition}
\label{betaCorresp}If $W_{1}=W\left( X_{1},Y_{1}\right) ,W_{2}=W\left(
X_{2},Y_{2}\right) \in \Xi ^{\prime }$, $H\in \Xi $, $\left(
T_{1},T_{2}\right) \in Cl_{H}\left( W_{2}\right) $ then $\mathcal{F}\left(
\beta _{W_{1},W_{2}}\left( T_{1},T_{2}\right) \right) =\beta _{\mathcal{F}%
\left( W_{1}\right) ,\mathcal{F}\left( W_{2}\right) }\left( \mathcal{F}%
_{W_{2},H}\left( T_{1},T_{2}\right) \right) $. If $F_{1}=F\left(
M_{1}\right) ,F_{2}=F\left( M_{2}\right) \in \mathrm{Ob}\Theta ^{0}$, $%
N=\left( N,p\right) \in \Theta $, $T\in Cl_{N}\left( F_{2}\right) $ then $%
\mathcal{F}^{-1}\left( \beta _{F_{1},F_{2}}\left( T\right) \right) =\beta _{%
\mathcal{F}^{-1}\left( F_{1}\right) ,\mathcal{F}^{-1}\left( F_{2}\right)
}\left( \mathcal{F}_{F_{2},N}^{-1}\left( T\right) \right) $.
\end{proposition}

\begin{proof}
$\mathcal{F}\left( W_{i}\right) =\left( L\left( X_{i}\right) \oplus A\left(
X_{i}\right) Y_{i},p_{\mid A\left( X_{i}\right) Y_{i}}\right) $ where $i=1,2$%
. $\mathcal{F}_{W_{2},H}\left( T_{1},T_{2}\right) =T_{1}\oplus
T_{2}\subseteq L\left( X_{2}\right) \oplus A\left( X_{2}\right) Y_{2}=%
\mathcal{F}\left( W_{2}\right) $. If $\left( \left( \varphi _{1},\psi
_{1}\right) ,\left( \varphi _{2},\psi _{2}\right) \right) \in \beta
_{W_{1},W_{2}}\left( T_{1},T_{2}\right) $ then $\varphi _{1}\left( l\right)
\equiv \varphi _{2}\left( l\right) \left( \func{mod}T_{1}\right) $ holds for
every $l\in L\left( X_{1}\right) $ and $\psi _{1}\left( v\right) \equiv \psi
_{2}\left( v\right) \left( \func{mod}T_{2}\right) $ holds for every $v\in
A\left( X_{1}\right) Y_{1}$. $\mathcal{F}\left( \varphi _{i},\psi
_{i}\right) =\varphi _{i}\oplus \psi _{i}\in \mathrm{Hom}\left( \mathcal{F}%
\left( W_{1}\right) ,\mathcal{F}\left( W_{2}\right) \right) $ where $i=1,2$.
For every $n\in \mathcal{F}\left( W_{1}\right) $ we have $n=l+v$, where $%
l\in L\left( X_{1}\right) $, $v\in A\left( X_{1}\right) Y_{1}$. So $\left(
\varphi _{1}\oplus \psi _{1}\right) \left( n\right) =\varphi _{1}\left(
l\right) +\psi _{1}\left( v\right) \equiv \varphi _{2}\left( l\right) +\psi
_{2}\left( v\right) \left( \func{mod}T_{1}\oplus T_{2}\right) $, $\varphi
_{2}\left( l\right) +\psi _{2}\left( v\right) =\left( \varphi _{2}\oplus
\psi _{2}\right) \left( n\right) $ and 
\begin{equation*}
\left( \mathcal{F}\left( \varphi _{1},\psi _{1}\right) ,\mathcal{F}\left(
\varphi _{2},\psi _{2}\right) \right) \in \beta _{\mathcal{F}\left(
W_{1}\right) ,\mathcal{F}\left( W_{2}\right) }\left( T_{1}\oplus
T_{2}\right) .
\end{equation*}

We assume that $\left( f_{1},f_{2}\right) \in \beta _{\mathcal{F}\left(
W_{1}\right) ,\mathcal{F}\left( W_{2}\right) }\left( T_{1}\oplus
T_{2}\right) $. $\mathcal{F}^{-1}\left( f_{i}\right) =\left( rf_{i}\kappa
,pf_{i}\iota \right) \in \mathrm{Hom}\left( W_{1},W_{2}\right) $ where $%
i=1,2 $. If $l\in L\left( X_{1}\right) $ then $f_{1}\left( l\right)
-f_{2}\left( l\right) \in T_{1}\oplus T_{2}$ and $rf_{1}\kappa \left(
l\right) -rf_{2}\kappa \left( l\right) \in T_{1}$. Analogously we have $%
pf_{1}\iota \left( v\right) -pf_{2}\iota \left( v\right) \in T_{2}$ for
every $v\in A\left( X_{2}\right) Y_{2}$, so%
\begin{equation*}
\left( \mathcal{F}^{-1}\left( f_{1}\right) ,\mathcal{F}^{-1}\left(
f_{2}\right) \right) =\left( \left( rf_{1}\kappa ,pf_{1}\iota \right)
,\left( rf_{2}\kappa ,pf_{2}\iota \right) \right) \in \beta
_{W_{1},W_{2}}\left( T_{1},T_{2}\right) .
\end{equation*}%
Therefore 
\begin{equation*}
\mathcal{F}\left( \beta _{W_{1},W_{2}}\left( T_{1},T_{2}\right) \right)
=\beta _{\mathcal{F}\left( W_{1}\right) ,\mathcal{F}\left( W_{2}\right)
}\left( \mathcal{F}_{W_{2},H}\left( T_{1},T_{2}\right) \right) .
\end{equation*}

From this fact and from proving of Proposition \ref{ClCorresp} we can
conclude that 
\begin{equation*}
\mathcal{F}^{-1}\left( \beta _{F_{1},F_{2}}\left( T\right) \right) =\beta _{%
\mathcal{F}^{-1}\left( F_{1}\right) ,\mathcal{F}^{-1}\left( F_{2}\right)
}\left( \mathcal{F}_{F_{2},N}^{-1}\left( T\right) \right) .
\end{equation*}
\end{proof}

\begin{theorem}
\label{automorphEquiv}If $H_{1}=\left( L_{1},V_{1}\right) ,H_{2}=\left(
L_{2},V_{2}\right) \in \Xi $ are automorphically equivalent then $N_{1}=%
\mathcal{F}\left( H_{1}\right) ,N_{2}=\mathcal{F}\left( H_{2}\right) $ are
automorphically equivalent.
\end{theorem}

\begin{proof}
We have an automorphism $\Phi \in \mathrm{Aut}\Xi ^{0}$ and the system of
bijections $\alpha \left( \Phi \right) _{W}:Cl_{H_{1}}\left( W\right)
\rightarrow Cl_{H_{2}}\left( \Phi \left( W\right) \right) $\ for every $W\in 
\mathrm{Ob}\Xi ^{0}$. Also the equation 
\begin{equation*}
\Phi \left( \beta _{W_{1},W_{2}}\left( T_{1},T_{2}\right) \right) =\beta
_{\Phi \left( W_{1}\right) ,\Phi \left( W_{2}\right) }\left( \alpha \left(
\Phi \right) _{W_{2}}\left( T_{1},T_{2}\right) \right)
\end{equation*}%
\ holds for every $W_{1},W_{2}\in \mathrm{Ob}\Xi ^{0}$, and every $\left(
T_{1},T_{2}\right) \in $ $Cl_{H_{1}}\left( W_{2}\right) $.

By Proposition \ref{withProjAutom} there is an automorphism $\Psi =\mathcal{F%
}\Phi _{\mid \Xi ^{\prime }}\mathcal{F}^{-1}\in \mathrm{Aut}\Theta ^{0}$. By
Proposition \ref{ClCorresp} the mapping: 
\begin{equation*}
\alpha \left( \Psi \right) _{F}=\mathcal{F}_{\Phi \mathcal{F}^{-1}\left(
F\right) ,H_{2}}\alpha \left( \Phi \right) _{\mathcal{F}^{-1}\left( F\right)
}\mathcal{F}_{F,N_{1}}^{-1}:Cl_{N_{1}}\left( F\right) \rightarrow
Cl_{N_{2}}\left( \Psi \left( F\right) \right)
\end{equation*}%
is a bijection for every $F\in \mathrm{Ob}\Theta ^{0}$. By Proposition \ref%
{betaCorresp} we have for every $F_{1},F_{2}\in \mathrm{Ob}\Theta ^{0}$, and
every $T\in $ $Cl_{N_{1}}\left( F_{2}\right) $ that 
\begin{equation*}
\Psi \left( \beta _{F_{1},F_{2}}\left( T\right) \right) =\mathcal{F}\Phi 
\mathcal{F}^{-1}\left( \beta _{F_{1},F_{2}}\left( T\right) \right) =\mathcal{%
F}\Phi \left( \beta _{\mathcal{F}^{-1}\left( F_{1}\right) ,\mathcal{F}%
^{-1}\left( F_{2}\right) }\left( \mathcal{F}_{F_{2},N_{1}}^{-1}\left(
T\right) \right) \right) =
\end{equation*}%
\begin{equation*}
\mathcal{F}\left( \beta _{\Phi \mathcal{F}^{-1}\left( F_{1}\right) ,\Phi 
\mathcal{F}^{-1}\left( F_{2}\right) }\left( \alpha \left( \Phi \right) _{%
\mathcal{F}^{-1}\left( F_{2}\right) }\left( \mathcal{F}_{F_{2},N_{1}}^{-1}%
\left( T\right) \right) \right) \right) =
\end{equation*}%
\begin{equation*}
\beta _{\mathcal{F}\Phi \mathcal{F}^{-1}\left( F_{1}\right) ,\mathcal{F}\Phi 
\mathcal{F}^{-1}\left( F_{2}\right) }\left( \mathcal{F}_{\Phi \mathcal{F}%
^{-1}\left( F_{2}\right) ,H_{2}}\alpha \left( \Phi \right) _{\mathcal{F}%
^{-1}\left( F_{2}\right) }\left( \mathcal{F}_{F_{2},N_{1}}^{-1}\left(
T\right) \right) \right) =
\end{equation*}%
\begin{equation*}
\beta _{\Psi \left( F_{1}\right) ,\Psi \left( F_{2}\right) }\left( \alpha
\left( \Psi \right) _{F_{2}}\left( T\right) \right) .
\end{equation*}
\end{proof}

\section{Automorphisms of the category of the finitely generated free
algebras of the some variety of $1$-sorted algebras.\label{1-sorted}}

\setcounter{equation}{0}

In this Section we explain the method of verbal operations which we will use
for the studying of the relation between the automorphic equivalence and
geometric equivalence in the our variety $\Theta $. We use results of the 
\cite{PlotkinZhitAutCat} and \cite{TsurkovAutomEquiv}.

In this Section the word "algebra" means "universal algebra". Also so on in
this Section $\Theta $ will be an arbitrary variety of $1$-sorted algebras.
As in the Section \ref{From2Autom_to_1Autom} we define the category $\Theta
^{0}$ of the finitely generated free algebras of our variety $\Theta $. The
infinite countable sets of symbols which will be the generators of our free
algebras we will denote in this Section by $X^{0}$.

\begin{definition}
An automorphism $\Upsilon $ of a category $\mathfrak{K}$ is \textbf{inner},
if it is isomorphic as a functor to the identity automorphism of the
category $\mathfrak{K}$.
\end{definition}

It means that for every $A\in \mathrm{Ob}\mathfrak{K}$ there exists an
isomorphism $s_{A}^{\Upsilon }:A\rightarrow \Upsilon \left( A\right) $ such
that for every $\alpha \in \mathrm{Mor}_{\mathfrak{K}}\left( A,B\right) $
the diagram%
\begin{equation*}
\begin{array}{ccc}
A & \overrightarrow{s_{A}^{\Upsilon }} & \Upsilon \left( A\right) \\ 
\downarrow \alpha &  & \Upsilon \left( \alpha \right) \downarrow \\ 
B & \underrightarrow{s_{B}^{\Upsilon }} & \Upsilon \left( B\right)%
\end{array}%
\end{equation*}%
\noindent commutes. The group of the all automorphisms of the category $%
\Theta ^{0}$ we denote by $\mathfrak{A}$. The subgroup of the all inner
automorphisms of $\Theta ^{0}$ we denote by $\mathfrak{Y}$. This is a normal
subgroup of $\mathfrak{A}$: $\mathfrak{Y}\vartriangleleft \mathfrak{A}$.

We know from \cite{PlotkinSame} that if automorphic equivalence of algebras $%
H_{1},H_{2}\in \Theta $ provided by inner automorphism then $H_{1}$ and $%
H_{2}$ are geometrically equivalent. Hear variety $\Theta $ can by even a
variety of many-sorted algebras. So for studying of the difference between
the automorphic equivalence and geometric equivalence of the algebras from $%
\Theta $, we must calculate the quotient group $\mathfrak{A/Y}$.

In the $1$-sorted case there is a reason to define

\begin{definition}
\label{str_stab_aut}\textbf{\hspace{-0.08in} }\textit{An automorphism $\Phi $
of the category }$\Theta ^{0}$\textit{\ is called \textbf{strongly stable}
if it satisfies the conditions:}

\begin{enumerate}
\item[A1)] $\Phi $\textit{\ preserves all objects of }$\Theta ^{0}$\textit{,}

\item[A2)] \textit{there exists a system of bijections }$\left\{ s_{F}^{\Phi
}:F\rightarrow F\mid F\in \mathrm{Ob}\Theta ^{0}\right\} $\textit{\ such
that }$\Phi $\textit{\ acts on the morphisms $\alpha :D\rightarrow F$ of }$%
\Theta ^{0}$\textit{\ by this way: }%
\begin{equation}
\Phi \left( \alpha \right) =s_{F}^{\Phi }\alpha \left( s_{D}^{\Phi }\right)
^{-1},  \label{action_by_bijections}
\end{equation}

\item[A3)] $s_{F}^{\Phi }\mid _{X}=id_{X},$ \textit{\ for every free algebra}
$F=F\left( X\right) \in \mathrm{Ob}\Theta ^{0}$.
\end{enumerate}
\end{definition}

The subgroup of the all strongly stable automorphisms of $\Theta ^{0}$ we
denote by $\mathfrak{S}$.

We say that the variety $\Theta $ has IBN propriety if for every $F\left(
X\right) ,F\left( Y\right) \in \mathrm{Ob}\Theta ^{0}$ we have $F\left(
X\right) \cong F\left( Y\right) $ only if $\left\vert X\right\vert
=\left\vert Y\right\vert $. In this case we have the decomposition%
\begin{equation}
\mathfrak{A=YS}  \label{decomp}
\end{equation}%
so $\mathfrak{A/Y=S/S\cap Y}$. The system of bijections $\left\{ s_{F}^{\Phi
}=s_{F}:F\rightarrow F\mid F\in \mathrm{Ob}\Theta ^{0}\right\} $ mentioned
in definition of the strongly stable automorphism fulfills these two
conditions:

\begin{enumerate}
\item[B1)] for every homomorphism $\alpha :A\rightarrow B\in \mathrm{Mor}%
\Theta ^{0}$ the mappings $s_{B}\alpha s_{A}^{-1}$ and $s_{B}^{-1}\alpha
s_{A}$ are homomorphisms;

\item[B2)] $s_{F}\mid _{X}=id_{X}$ for every free algebra $F\in \mathrm{Ob}%
\Theta ^{0}$.
\end{enumerate}

These bijections uniquely defined by the strongly stable automorphism $\Phi $%
, because for every $F\in \mathrm{Ob}\Theta ^{0}$ and every $f\in F$ we have%
\begin{equation}
s_{F}^{\Phi }\left( f\right) =s_{F}^{\Phi }\alpha \left( x\right) =\left(
s_{F}^{\Phi }\alpha \left( s_{D}^{\Phi }\right) ^{-1}\right) \left( x\right)
=\left( \Phi \left( \alpha \right) \right) \left( x\right) ,
\label{autom_bijections}
\end{equation}%
where $D=D\left( x\right) \in \mathrm{Ob}\Theta ^{0}$ is a $1$-generated
free algebra and $\alpha :D\rightarrow F$ homomorphism such that $\alpha
\left( x\right) =f$.

On the other side by system of bijections $\left\{ s_{F}:F\rightarrow F\mid
F\in \mathrm{Ob}\Theta ^{0}\right\} $ which fulfills conditions B1) and B2)
we can define the strongly stable automorphism $\Phi $, which preserves all
objects of $\Theta ^{0}$ and acts on the morphisms $\alpha :D\rightarrow F$
of $\Theta ^{0}$\ by formula (\ref{action_by_bijections}) with $s_{F}^{\Phi
}=s_{F}$. By this way we construct an one-to-one and onto correspondence
between the set of the all strongly stable automorphisms of the category $%
\Theta ^{0}$ and the set of the all systems of bijections which fulfill
conditions B1) and B2).

We denote the signature of the algebras from the variety $\Theta $ by $%
\Omega $. The arity of the operation $\omega \in \Omega $ we denote by $%
n_{\omega }$ and by $F_{\omega }$ we denote $F\left( x_{1},\ldots
,x_{n_{\omega }}\right) \in \mathrm{Ob}\Theta ^{0}$. $\omega \left(
x_{1},\ldots ,x_{n_{\omega }}\right) \in F_{\omega }$. If we have system of
bijections $\left\{ s_{F}:F\rightarrow F\mid F\in \mathrm{Ob}\Theta
^{0}\right\} $ which fulfills conditions B1) and B2) then 
\begin{equation}
w_{\omega }\left( x_{1},\ldots ,x_{n_{\omega }}\right) =s_{F_{\omega
}}\left( \omega \left( x_{1},\ldots ,x_{n_{\omega }}\right) \right) \in
F_{\omega }.  \label{word_definition}
\end{equation}%
We will consider the system of words $W=\left\{ w_{\omega }\mid \omega \in
\Omega \right\} $. In every $H\in \Theta $ we can define new operations $%
\left\{ \omega ^{\ast }\mid \omega \in \Omega \right\} $ by using of the
system of words $W$: 
\begin{equation}
\omega ^{\ast }\left( h_{1},\ldots ,h_{n_{\omega }}\right) =w_{\omega
}\left( h_{1},\ldots ,h_{n_{\omega }}\right)  \label{def_operation}
\end{equation}%
for every $h_{1},\ldots ,h_{n_{\omega }}\in H$. We denote by $H_{W}^{\ast }$
the new algebra which coincide as set with $H$ but has other operations: $%
\left\{ \omega ^{\ast }\mid \omega \in \Omega \right\} $ instead $\left\{
\omega \mid \omega \in \Omega \right\} $. The system of words $W=\left\{
w_{\omega }\mid \omega \in \Omega \right\} $ fulfills these two conditions:

\begin{enumerate}
\item[Op1)] $w_{\omega }\left( x_{1},\ldots ,x_{n_{\omega }}\right) \in
F_{\omega }$ for every $\omega \in \Omega $,

\item[Op2)] for every $F=F\left( X\right) \in \mathrm{Ob}\Theta ^{0}$ there
exists an isomorphism $\sigma _{F}:F\rightarrow F_{W}^{\ast }$ such that $%
\sigma _{F}\mid _{X}=id_{X}$
\end{enumerate}

because the bijections $\left\{ s_{F}\mid F\in \mathrm{Ob}\Theta
^{0}\right\} $ will be isomorphisms $\sigma _{F}:F\rightarrow F_{W}^{\ast }$.

On the other side if we have a system of words $W=\left\{ w_{\omega }\mid
\omega \in \Omega \right\} $ which fulfills conditions Op1) and Op2), then
we have that $F_{W}^{\ast }\in \Theta $, so the isomorphisms $\sigma
_{F}:F\rightarrow F_{W}^{\ast }$ are uniquely determined by the system of
words $W$. This system of isomorphisms $\left\{ \sigma _{F}:F\rightarrow
F_{W}^{\ast }\mid F\in \mathrm{Ob}\Theta ^{0}\right\} $ is a system of
bijections which fulfills conditions B1) and B2) with $s_{F}=\sigma _{F}$.
By this way we construct an one-to-one and onto correspondence between the
set of the all system of bijections which fulfills conditions B1) and B2)
and the set of the all system of words which fulfills conditions Op1) and
Op2).

Therefore we can calculate the group $\mathfrak{S}$ if we can find the all
system of words which fulfill conditions Op1) and Op2). For calculation of
the group $\mathfrak{S\cap Y}$ we also have a

\begin{criterion}
\label{inner_stable}The strongly stable automorphism $\Phi $ of the category 
$\Theta ^{0}$ which corresponds to the system of words $W$ is inner if and
only if for every $F\in \mathrm{Ob}\Theta ^{0}$ there exists an isomorphism $%
c_{F}:F\rightarrow F_{W}^{\ast }$ such that $c_{F}\alpha =\alpha c_{D}$
fulfills for every $\left( \alpha :D\rightarrow F\right) \in \mathrm{Mor}%
\Theta ^{0}$.
\end{criterion}

\section{Strongly stable automorphisms of the category $\Theta ^{0}$.}

\setcounter{equation}{0}

The variety $\Theta $ is a variety of $1$-sorted universal algebras. If $%
F\left( M\right) \in \mathrm{Ob}\Theta ^{0}$ then by Theorem \ref%
{freeWithProj} $\left\vert M\right\vert =\dim \left( \ker p/\left( \ker
p\right) ^{2}\right) $, so variety $\Theta $ possesses the IBN property: for
free algebras $F\left( M_{1}\right) ,F\left( M_{2}\right) \in \Theta $ we
have $F\left( M_{1}\right) \cong F\left( M_{2}\right) $ if and only if $%
\left\vert M_{1}\right\vert =\left\vert M_{2}\right\vert $. So we have for
our variety $\Theta ^{0}$ the decomposition (\ref{decomp}) and for
calculation of the group $\mathfrak{A/Y=S/S\cap Y}$ we can use the method
described in the Section \ref{1-sorted}.

The signature of our variety $\Theta $ is $\Omega =\left\{ 0,\lambda \left(
\lambda \in k\right) ,+,\left[ ,\right] ,p\right\} $, where $0$ is $0$-ary
operation of the taking $0$, $\lambda $ for every $\lambda \in k$ is the $1$%
-nary operation of the multiplication by this scalar, $p$ is the $1$-nary
operation of projection, $+$ is the addition and $\left[ ,\right] $ are the
Lie brackets. We must find for the calculation of the group $\mathfrak{S}$
all the system of words 
\begin{equation}
W=\left\{ w_{0},w_{\lambda }\left( \lambda \in k\right) ,w_{+},w_{\left[ ,%
\right] },w_{p}\right\}  \label{words_list}
\end{equation}
which fulfill conditions Op1) and Op2) and after use the Criterion \ref%
{inner_stable} for the calculation of the group $\mathfrak{S\cap Y}$. By
this way we will prove the

\begin{theorem}
\label{group}If $\mathrm{Aut}k=\left\{ id_{k}\right\} $ then the group $%
\mathfrak{A/Y}$ is a trivial.
\end{theorem}

\begin{proof}
If $\left( F,p\right) =F\left( m_{1},\ldots ,m_{n}\right) \in \mathrm{Ob}%
\Theta ^{0}$ then by Theorems \ref{freeWithProj} and \ref{revFreeWithProj} $%
\left( F,p\right) =\mathcal{FF}^{-1}\left( F,p\right) =\left( L\oplus
V,p_{V}\right) $, where $p=p_{V}$, 
\begin{equation*}
L=\ker p=L\left( r\left( m_{1}\right) ,\ldots ,r\left( m_{n}\right) \right)
\end{equation*}%
is a free Lie algebra with the free generators $r\left( m_{1}\right) ,\ldots
,r\left( m_{n}\right) $, 
\begin{equation*}
V=\mathrm{im}p=\bigoplus\limits_{i=1}^{n}A\left( r\left( m_{1}\right)
,\ldots ,r\left( m_{n}\right) \right) p\left( m_{i}\right)
\end{equation*}%
is a free module with the basis $\left\{ p\left( m_{1}\right) ,\ldots
,p\left( m_{n}\right) \right\} $ over algebra $A\left( r\left( m_{1}\right)
,\ldots ,r\left( m_{n}\right) \right) $, which is a associative algebra with
unit generated by the free generators $r\left( m_{1}\right) ,\ldots ,r\left(
m_{n}\right) $. Hear we must understand that by formula (\ref{Maction}) 
\begin{equation*}
r\left( m_{i_{1}}\right) \ldots r\left( m_{i_{s}}\right) v=\left[ r\left(
m_{i_{1}}\right) ,\left[ \ldots ,\left[ r\left( m_{i_{s}}\right) ,v\right] %
\right] \right] ,
\end{equation*}%
where $v\in V$, $1\leq i_{1},\ldots ,i_{s}\leq n$, $s\in 
\mathbb{N}
$, if $s=0$ then $1v=v$. So by linearity we can understand what means $%
f\left( r\left( m_{1}\right) ,\ldots ,r\left( m_{n}\right) \right) v$, for
every associative polynomial from $n$ variables $f\in A\left( x_{1},\ldots
,x_{n}\right) $.

We assume that $\Psi \in \mathfrak{S}$ corresponds to the system of
bijections $\left\{ s_{F}^{\Psi }=s_{F}:F\rightarrow F\mid F\in \mathrm{Ob}%
\Theta ^{0}\right\} $ and to the system of words (\ref{words_list}) and the
words of this system correspond to the operations from $\Omega $ by formula (%
\ref{word_definition}) with $s_{F_{\omega }}=s_{F_{\omega }}^{\Psi }$. $W$
fulfills conditions Op1) and Op2). In particular by condition Op2) all
axioms of the variety $\Theta $ must fulfill for operations defined by
system of words $W$. In this proof we have more convenient to denote by an
other symbols than the symbols of $\Omega $ the operations defined by the
words from $W$ according the (\ref{def_operation}).

$w_{0}=0$ because $w_{0}\in F\left( \varnothing \right) $ and $F\left(
\varnothing \right) =\left\{ 0\right\} $.

We denote by $\lambda \ast $ the operation defined by the word $w_{\lambda
}\in F\left( m\right) $ $\left( \lambda \in k\right) $, where $F\left(
m\right) $ is a $1$-generated object of the category $\Theta ^{0}$. 
\begin{equation*}
\lambda \ast m=w_{\lambda }\left( m\right) =\varphi \left( \lambda \right)
r\left( m\right) +q_{\lambda }\left( r\left( m\right) \right) p\left(
m\right) ,
\end{equation*}%
where $\varphi \left( \lambda \right) \in k$, $q_{\lambda }\in k\left[ x%
\right] $. If $\lambda \neq 0$ then $\lambda ^{-1}\ast \left( \lambda \ast
m\right) =m$ must fulfill. 
\begin{equation*}
\lambda ^{-1}\ast \left( \lambda \ast m\right) =\lambda ^{-1}\ast \left(
\varphi \left( \lambda \right) r\left( m\right) +q_{\lambda }\left( r\left(
m\right) \right) p\left( m\right) \right) =
\end{equation*}%
\begin{equation*}
\varphi \left( \lambda ^{-1}\right) r\left( \varphi \left( \lambda \right)
r\left( m\right) +q_{\lambda }\left( r\left( m\right) \right) p\left(
m\right) \right) +
\end{equation*}%
\begin{equation*}
q_{\lambda ^{-1}}\left( r\left( \varphi \left( \lambda \right) r\left(
m\right) +q_{\lambda }\left( r\left( m\right) \right) p\left( m\right)
\right) \right) p\left( \varphi \left( \lambda \right) r\left( m\right)
+q_{\lambda }\left( r\left( m\right) \right) p\left( m\right) \right) =
\end{equation*}%
\begin{equation*}
\varphi \left( \lambda ^{-1}\right) \varphi \left( \lambda \right) r\left(
m\right) +q_{\lambda ^{-1}}\left( \varphi \left( \lambda \right) r\left(
m\right) \right) \left( q_{\lambda }\left( r\left( m\right) \right) p\left(
m\right) \right) .
\end{equation*}%
On the other side $m=r\left( m\right) +p\left( m\right) $. So $\varphi
\left( \lambda ^{-1}\right) \varphi \left( \lambda \right) =1$ and $\varphi
\left( \lambda \right) \neq 0$. 
\begin{equation*}
q_{\lambda ^{-1}}\left( \varphi \left( \lambda \right) r\left( m\right)
\right) \left( q_{\lambda }\left( r\left( m\right) \right) p\left( m\right)
\right) =s\left( r\left( m\right) \right) p\left( m\right) ,
\end{equation*}%
where $s\in k\left[ x\right] $. If $\deg q_{\lambda ^{-1}}=n$, $\deg
q_{\lambda }=t$ then $\deg s=n+t$, but $\deg s=0$ must hold, so $n=0$, $t=0$%
. Therefore $q_{\lambda }=\psi \left( \lambda \right) \in k$, 
\begin{equation}
\lambda \ast m=\varphi \left( \lambda \right) r\left( m\right) +\psi \left(
\lambda \right) p\left( m\right) .  \label{scalar_mult_word}
\end{equation}%
For $\lambda =0$ it also fulfills with $\varphi \left( 0\right) =\psi \left(
0\right) =0$, because $0\ast m=0$.

$\mu \ast \left( \lambda \ast m\right) =\left( \mu \lambda \right) \ast m$
must fulfill for every $\mu ,\lambda \in k$ so 
\begin{equation*}
\mu \ast \left( \lambda \ast m\right) =\mu \ast \left( \varphi \left(
\lambda \right) r\left( m\right) +\psi \left( \lambda \right) p\left(
m\right) \right) =
\end{equation*}%
\begin{equation*}
\varphi \left( \mu \right) r\left( \varphi \left( \lambda \right) r\left(
m\right) +\psi \left( \lambda \right) p\left( m\right) \right) +\psi \left(
\mu \right) p\left( \varphi \left( \lambda \right) r\left( m\right) +\psi
\left( \lambda \right) p\left( m\right) \right) =
\end{equation*}%
\begin{equation*}
\varphi \left( \mu \right) \varphi \left( \lambda \right) r\left( m\right)
+\psi \left( \mu \right) \psi \left( \lambda \right) p\left( m\right) .
\end{equation*}%
On the other side 
\begin{equation*}
\left( \mu \lambda \right) \ast m=\varphi \left( \mu \lambda \right) r\left(
m\right) +\psi \left( \mu \lambda \right) p\left( m\right) .
\end{equation*}%
Hence 
\begin{equation}
\varphi \left( \mu \right) \varphi \left( \lambda \right) =\varphi \left(
\mu \lambda \right) ,\psi \left( \mu \right) \psi \left( \lambda \right)
=\psi \left( \mu \lambda \right) .  \label{mult_hom}
\end{equation}

We denote by $\perp $ the operation defined by the word $w_{+}\in F\left(
m_{1},m_{2}\right) $, where $F\left( m_{1},m_{2}\right) $ is a $2$-generated
object of the category $\Theta ^{0}$. 
\begin{equation*}
m_{1}\perp m_{2}=l\left( r\left( m_{1}\right) ,r\left( m_{2}\right) \right)
+q_{1}\left( r\left( m_{1}\right) ,r\left( m_{2}\right) \right) p\left(
m_{1}\right) +q_{2}\left( r\left( m_{1}\right) ,r\left( m_{2}\right) \right)
p\left( m_{2}\right) ,
\end{equation*}%
where $l\in L\left( x_{1},x_{2}\right) $, $q_{1},q_{2}\in A\left(
x_{1},x_{2}\right) $. We can write 
\begin{equation*}
l\left( r\left( m_{1}\right) ,r\left( m_{2}\right) \right) =\alpha
_{1}r\left( m_{1}\right) +\alpha _{2}r\left( m_{2}\right) +\widetilde{l}%
\left( r\left( m_{1}\right) ,r\left( m_{2}\right) \right) ,
\end{equation*}%
where $\widetilde{l}\in L^{2}\left( x_{1},x_{2}\right) $, $\alpha
_{1},\alpha _{2}\in k$. And 
\begin{equation*}
q_{i}\left( r\left( m_{1}\right) ,r\left( m_{2}\right) \right) p\left(
m_{i}\right) =\widetilde{q}_{i}\left( r\left( m_{1}\right) ,r\left(
m_{2}\right) \right) p\left( m_{i}\right) +\beta _{i}p\left( m_{i}\right) ,
\end{equation*}%
where $\widetilde{q}_{i}$ is a polynomial from $A\left( x_{1},x_{2}\right) $
such that all its monomials have entries of $x_{1}$ or $x_{2}$, $\beta
_{i}\in k$, $i=1,2$.

$m_{1}\perp 0=m_{1}$ must fulfill. $m_{1}=r\left( m_{1}\right) +p\left(
m_{1}\right) $. But 
\begin{equation*}
m_{1}\perp 0=\alpha _{1}r\left( m_{1}\right) +\widetilde{q}_{1}\left(
r\left( m_{1}\right) ,0\right) p\left( m_{1}\right) +\beta _{1}p\left(
m_{1}\right) .
\end{equation*}%
Therefore $\alpha _{1}=\beta _{1}=1$. From $0\perp m_{2}=m_{2}$ we conclude
that $\alpha _{2}=\beta _{2}=1$.

In $F\left( m\right) $ the $\left( \lambda +\mu \right) \ast m=\left(
\lambda \ast m\right) \perp \left( \mu \ast m\right) $ must fulfill for
every $\mu ,\lambda \in k$.%
\begin{equation*}
\left( \lambda +\mu \right) \ast m=\varphi \left( \lambda +\mu \right)
r\left( m\right) +\psi \left( \lambda +\mu \right) p\left( m\right) .
\end{equation*}%
Also we have that 
\begin{equation*}
\left( \lambda \ast m\right) \perp \left( \mu \ast m\right) =r\left( \lambda
\ast m\right) +r\left( \mu \ast m\right) +\widetilde{l}\left( r\left(
\lambda \ast m\right) ,r\left( \mu \ast m\right) \right) +
\end{equation*}%
\begin{equation*}
\widetilde{q}_{1}\left( r\left( \lambda \ast m\right) ,r\left( \mu \ast
m\right) \right) p\left( \lambda \ast m\right) +p\left( \lambda \ast
m\right) +
\end{equation*}%
\begin{equation*}
\widetilde{q}_{2}\left( r\left( \lambda \ast m\right) ,r\left( \mu \ast
m\right) \right) p\left( \mu \ast m\right) +p\left( \mu \ast m\right) .
\end{equation*}%
\begin{equation*}
r\left( \lambda \ast m\right) =r\left( \varphi \left( \lambda \right)
r\left( m\right) +\psi \left( \lambda \right) p\left( m\right) \right)
=\varphi \left( \lambda \right) r\left( m\right) ,
\end{equation*}%
\begin{equation*}
p\left( \lambda \ast m\right) =p\left( \varphi \left( \lambda \right)
r\left( m\right) +\psi \left( \lambda \right) p\left( m\right) \right) =\psi
\left( \lambda \right) p\left( m\right) .
\end{equation*}%
So%
\begin{equation*}
\left( \lambda \ast m\right) \perp \left( \mu \ast m\right) =\varphi \left(
\lambda \right) r\left( m\right) +\varphi \left( \mu \right) r\left(
m\right) +\widetilde{l}\left( \varphi \left( \lambda \right) r\left(
m\right) ,\varphi \left( \mu \right) r\left( m\right) \right) +
\end{equation*}%
\begin{equation*}
\widetilde{q}_{1}\left( \varphi \left( \lambda \right) r\left( m\right)
,\varphi \left( \mu \right) r\left( m\right) \right) \psi \left( \lambda
\right) p\left( m\right) +\psi \left( \lambda \right) p\left( m\right) +
\end{equation*}%
\begin{equation*}
\widetilde{q}_{2}\left( \varphi \left( \lambda \right) r\left( m\right)
,\varphi \left( \mu \right) r\left( m\right) \right) \psi \left( \mu \right)
p\left( m\right) +\psi \left( \mu \right) p\left( m\right) .
\end{equation*}%
Hence 
\begin{equation}
\varphi \left( \lambda +\mu \right) =\varphi \left( \lambda \right) +\varphi
\left( \mu \right) ,\psi \left( \lambda +\mu \right) =\psi \left( \lambda
\right) +\psi \left( \mu \right) .  \label{add_hom}
\end{equation}%
Therefore $\varphi ,\psi $ are homomorphisms $k\rightarrow k$. 
\begin{equation}
\varphi \left( 1\right) =\psi \left( 1\right) =1,  \label{monomorph}
\end{equation}%
because $1\ast m=m$, so $\ker \varphi =\ker \psi =0$ and $\func{Im}\varphi
\cong \func{Im}\psi \cong k$. Hence $\func{Im}\varphi ,\func{Im}\psi $ are
infinite sets.

$\lambda \ast \left( m_{1}\perp m_{2}\right) =\left( \lambda \ast
m_{1}\right) \perp \left( \lambda \ast m_{2}\right) $ must fulfill for every 
$\lambda \in k$.%
\begin{equation*}
\lambda \ast \left( m_{1}\perp m_{2}\right) =\varphi \left( \lambda \right)
\left( r\left( m_{1}\right) +r\left( m_{2}\right) +\widetilde{l}\left(
r\left( m_{1}\right) ,r\left( m_{2}\right) \right) \right) +
\end{equation*}%
\begin{equation}
\psi \left( \lambda \right) \left( \widetilde{q}_{1}\left( r\left(
m_{1}\right) ,r\left( m_{2}\right) \right) p\left( m_{1}\right) +p\left(
m_{1}\right) \right) +  \label{add1}
\end{equation}%
\begin{equation*}
\psi \left( \lambda \right) \left( \widetilde{q}_{2}\left( r\left(
m_{1}\right) ,r\left( m_{2}\right) \right) p\left( m_{2}\right) +p\left(
m_{2}\right) \right) .
\end{equation*}%
\begin{equation*}
\left( \lambda \ast m_{1}\right) \perp \left( \lambda \ast m_{2}\right)
=r\left( \lambda \ast m_{1}\right) +r\left( \lambda \ast m_{2}\right) +%
\widetilde{l}\left( r\left( \lambda \ast m_{1}\right) ,r\left( \lambda \ast
m_{2}\right) \right) +
\end{equation*}%
\begin{equation*}
\widetilde{q}_{1}\left( r\left( \lambda \ast m_{1}\right) ,r\left( \lambda
\ast m_{2}\right) \right) p\left( \lambda \ast m_{1}\right) +p\left( \lambda
\ast m_{1}\right) +
\end{equation*}%
\begin{equation*}
\widetilde{q}_{2}\left( r\left( \lambda \ast m_{1}\right) ,r\left( \lambda
\ast m_{2}\right) \right) p\left( \lambda \ast m_{2}\right) +p\left( \lambda
\ast m_{2}\right) =
\end{equation*}%
\begin{equation*}
\varphi \left( \lambda \right) r\left( m_{1}\right) +\varphi \left( \lambda
\right) r\left( m_{2}\right) +\widetilde{l}\left( \varphi \left( \lambda
\right) r\left( m_{1}\right) ,\varphi \left( \lambda \right) r\left(
m_{2}\right) \right) +
\end{equation*}%
\begin{equation}
\widetilde{q}_{1}\left( \varphi \left( \lambda \right) r\left( m_{1}\right)
,\varphi \left( \lambda \right) r\left( m_{2}\right) \right) \psi \left(
\lambda \right) p\left( m_{1}\right) +\psi \left( \lambda \right) p\left(
m_{1}\right) +  \label{add2}
\end{equation}%
\begin{equation*}
\widetilde{q}_{2}\left( \varphi \left( \lambda \right) r\left( m_{1}\right)
,\varphi \left( \lambda \right) r\left( m_{2}\right) \right) \psi \left(
\lambda \right) p\left( m_{2}\right) +\psi \left( \lambda \right) p\left(
m_{2}\right) .
\end{equation*}%
We decompose $\widetilde{l}$, $\widetilde{q}_{1}$ and $\widetilde{q}_{2}$ to
the homogeneous components according the degrees (sum of degrees of
variables $x_{1}$ and $x_{2}$) of monomials: $\widetilde{l}=l_{2}+\ldots
+l_{n_{0}}$, $\widetilde{q}_{i}=q_{i,1}+\ldots +q_{i,n_{i}}$, $n_{0}=\deg 
\widetilde{l}$, $n_{i}=\deg \widetilde{q}_{i}$, $i=1,2$. We have by
comparison of (\ref{add1}) and (\ref{add2}) that $\varphi \left( \lambda
\right) l_{j}=\left( \varphi \left( \lambda \right) \right) ^{j}l_{j}$ for $%
2\leq j\leq n_{0}$ and $\psi \left( \lambda \right) q_{i,j}=\psi \left(
\lambda \right) \left( \varphi \left( \lambda \right) \right) ^{j}q_{i,j}$
for $1\leq j\leq n_{i}$, $i=1,2$. We denote $n=\max \left\{
n_{0},n_{1},n_{2}\right\} $. We take $\mu =\varphi \left( \lambda \right)
\in \func{Im}\varphi \setminus \left\{ 0\right\} $ such that $\varphi \left(
\lambda \right) ^{j}\neq 1$ for every $j=1,\ldots ,n$. $\psi \left( \lambda
\right) \neq 0$, so $l_{j}=0$, $q_{i,j}=0$, hence $\widetilde{l}=0$, $%
\widetilde{q}_{1}=\widetilde{q}_{2}=0$ and%
\begin{equation}
m_{1}\perp m_{2}=r\left( m_{1}\right) +r\left( m_{2}\right) +p\left(
m_{1}\right) +p\left( m_{2}\right) =m_{1}+m_{2}.  \label{addition}
\end{equation}

We denote by $W^{\Psi ^{-1}}$ the system of words which fulfills conditions
Op1) and Op2) and corresponds to the automorphism $\Psi ^{-1}$. By $%
w_{\lambda }^{\Psi ^{-1}}\left( m\right) $ ($\lambda \in k$) we denote the
word from $W^{\Psi ^{-1}}$ which corresponds to the operation of the
multiplication by the scalar $\lambda $. We denote by $\lambda \underset{%
\Psi ^{-1}}{\ast }$ the operation defined by word $w_{\lambda }^{\Psi
^{-1}}\left( m\right) $. By (\ref{scalar_mult_word}), (\ref{mult_hom}), (\ref%
{add_hom}) and (\ref{monomorph}) $w_{\lambda }^{\Psi ^{-1}}\left( m\right)
=\rho \left( \lambda \right) r\left( m\right) +\sigma \left( \lambda \right)
p\left( m\right) $, where $\rho ,\sigma $ are monomorphisms of the field $k$%
. By (\ref{addition}) for $w_{+}$ we have only one possibility for every
system of words which fulfills conditions Op1) and Op2): $w_{+}\left(
m_{1},m_{2}\right) =m_{1}+m_{2}$.

By $\left\{ s_{F}^{\Psi ^{-1}}\right\} $ we denote the systems of bijections
corresponding to automorphism $\Psi ^{-1}$. $\Psi ^{-1}\Psi =\Psi \Psi
^{-1}=I$, where $I$ is the identical automorphism . By consideration of the
formula (\ref{action_by_bijections}) we can conclude that to the
automorphism $\Psi ^{-1}\Psi $ corresponds the systems of bijections $%
\left\{ s_{F}^{\Psi ^{-1}}s_{F}^{\Psi }\mid F\in \mathrm{Ob}\Theta
^{0}\right\} $. On the other side to the automorphism $I$ corresponds the
systems of bijections $\left\{ id_{F}\mid F\in \mathrm{Ob}\Theta
^{0}\right\} $. So, we have 
\begin{equation*}
s_{F\left( m\right) }^{\Psi ^{-1}}s_{F\left( m\right) }^{\Psi }\left(
\lambda m\right) =s_{F\left( m\right) }^{I}\left( \lambda m\right) =\lambda
m=\lambda r\left( m\right) +\lambda p\left( m\right) .
\end{equation*}%
On the other side, by using of the formula (\ref{word_definition}),%
\begin{equation*}
s_{F\left( m\right) }^{\Psi ^{-1}}s_{F\left( m\right) }^{\Psi }\left(
\lambda m\right) =s_{F\left( m\right) }^{\Psi ^{-1}}\left( \varphi \left(
\lambda \right) r\left( m\right) +\psi \left( \lambda \right) p\left(
m\right) \right) =
\end{equation*}%
\begin{equation*}
\varphi \left( \lambda \right) \underset{\Psi ^{-1}}{\ast }r\left( m\right)
+\psi \left( \lambda \right) \underset{\Psi ^{-1}}{\ast }p\left( m\right) =
\end{equation*}%
\begin{equation*}
\left( \rho \varphi \left( \lambda \right) rr\left( m\right) +\sigma \varphi
\left( \lambda \right) pr\left( m\right) \right) +\left( \rho \psi \left(
\lambda \right) rp\left( m\right) +\sigma \psi \left( \lambda \right)
pp\left( m\right) \right) =
\end{equation*}%
\begin{equation*}
=\rho \varphi \left( \lambda \right) r\left( m\right) +\sigma \psi \left(
\lambda \right) p\left( m\right) .
\end{equation*}%
Therefore $\rho \varphi =\sigma \psi =id_{k}$. Analogously $\varphi \rho
=\psi \sigma =id_{k}$. Therefore $\varphi ,\psi \in \mathrm{Aut}k$.

Now we consider the case when $\mathrm{Aut}k=\left\{ id_{k}\right\} $. We
denote by $\times $ the operation defined by the word $w_{\left[ ,\right]
}\in F\left( m_{1},m_{2}\right) $.%
\begin{equation*}
m_{1}\times m_{2}=
\end{equation*}%
\begin{equation*}
u\left( r\left( m_{1}\right) ,r\left( m_{2}\right) \right) +t_{1}\left(
r\left( m_{1}\right) ,r\left( m_{2}\right) \right) p\left( m_{1}\right)
+t_{2}\left( r\left( m_{1}\right) ,r\left( m_{2}\right) \right) p\left(
m_{2}\right) ,
\end{equation*}%
where $u\in L\left( x_{1},x_{2}\right) $, $t_{1},t_{2}\in A\left(
x_{1},x_{2}\right) $. $\left( \lambda m_{1}\right) \times m_{2}=\lambda
\left( m_{1}\times m_{2}\right) $ must fulfill for every $\lambda \in k$.%
\begin{equation*}
\lambda \left( m_{1}\times m_{2}\right) =\lambda u\left( r\left(
m_{1}\right) ,r\left( m_{2}\right) \right) +\lambda t_{1}\left( r\left(
m_{1}\right) ,r\left( m_{2}\right) \right) p\left( m_{1}\right) +
\end{equation*}%
\begin{equation}
\lambda t_{2}\left( r\left( m_{1}\right) ,r\left( m_{2}\right) \right)
p\left( m_{2}\right) .  \label{lie1}
\end{equation}%
\begin{equation*}
\left( \lambda m_{1}\right) \times m_{2}=u\left( \lambda r\left(
m_{1}\right) ,r\left( m_{2}\right) \right) +t_{1}\left( \lambda r\left(
m_{1}\right) ,r\left( m_{2}\right) \right) \lambda p\left( m_{1}\right) +
\end{equation*}%
\begin{equation}
t_{2}\left( \lambda r\left( m_{1}\right) ,r\left( m_{2}\right) \right)
p\left( m_{2}\right) .  \label{lie2}
\end{equation}%
We decompose $u=u_{0}+u_{1}+\ldots +u_{s_{0}}$, $t_{i}=t_{i,0}+t_{i,1}+%
\ldots +t_{i,s_{i}}$, $i=1,2$ by homogeneous components according the degree
of $x_{1}$. By comparison of (\ref{lie1}) and (\ref{lie2}) we have that $%
\lambda u_{j}=\lambda ^{j}u_{j}$ for $0\leq j\leq $ $s_{0}$, $\lambda
t_{1,j}=\lambda ^{j+1}t_{1,j}$, for $0\leq j\leq $ $s_{1}$, $\lambda
t_{2,j}=\lambda ^{j}t_{2,j}$, for $0\leq j\leq $ $s_{2}$. We denote $s=\max
\left\{ s_{0},s_{1},s_{2}\right\} $. We take $\lambda $ such that $\lambda
^{j}\neq \lambda $ for $j=0,2,\ldots ,s+1$ and conclude that $u=u_{1}$, $%
t_{1}=t_{1,0}$, $t_{2}=t_{2,1}$.

Also $m_{1}\times \left( \lambda m_{2}\right) =\lambda \left( m_{1}\times
m_{2}\right) $ must fulfill for every $\lambda \in k$. 
\begin{equation*}
m_{1}\times \left( \lambda m_{2}\right) =u_{1}\left( r\left( m_{1}\right)
,\lambda r\left( m_{2}\right) \right) +t_{1,0}\left( r\left( m_{1}\right)
,\lambda r\left( m_{2}\right) \right) p\left( m_{1}\right) +
\end{equation*}%
\begin{equation}
t_{2,1}\left( r\left( m_{1}\right) ,\lambda r\left( m_{2}\right) \right)
\lambda p\left( m_{2}\right) .  \label{lie3}
\end{equation}%
Now we decompose $u_{1}=u_{1,0}+u_{1,1}+\ldots +u_{1,s_{3}}$, $%
t_{1,0}=t_{1,0,0}+t_{1,0,1}+\ldots +t_{1,0,s_{4}}$, $%
t_{2,1}=t_{2,1,0}+t_{2,1,1}+\ldots +t_{2,1,s_{5}}$, by homogeneous
components according the degree of $x_{2}$. And by comparison of (\ref{lie1}%
) and (\ref{lie3}) as above we conclude that $u=u_{1}=u_{1,1}$, $%
t_{1}=t_{1,0}=t_{1,0,1}$, $t_{2}=t_{2,1}=t_{2,1,0}$. Therefore by (\ref%
{Maction})%
\begin{equation*}
m_{1}\times m_{2}=\alpha \left[ r\left( m_{1}\right) ,r\left( m_{2}\right) %
\right] +\beta \left[ r\left( m_{2}\right) ,p\left( m_{1}\right) \right]
+\gamma \left[ r\left( m_{1}\right) ,p\left( m_{2}\right) \right] ,
\end{equation*}%
where $\alpha ,\beta ,\gamma \in k$.

$m_{1}\times m_{2}=-m_{2}\times m_{1}$ must fulfill.%
\begin{equation*}
m_{2}\times m_{1}=\alpha \left[ r\left( m_{2}\right) ,r\left( m_{1}\right) %
\right] +\beta \left[ r\left( m_{1}\right) ,p\left( m_{2}\right) \right]
+\gamma \left[ r\left( m_{2}\right) ,p\left( m_{1}\right) \right] =
\end{equation*}%
\begin{equation*}
-\alpha \left[ r\left( m_{1}\right) ,r\left( m_{2}\right) \right] +\gamma %
\left[ r\left( m_{2}\right) ,p\left( m_{1}\right) \right] +\beta \left[
r\left( m_{1}\right) ,p\left( m_{2}\right) \right] .
\end{equation*}
Therefore $\gamma =-\beta $ and%
\begin{equation*}
m_{1}\times m_{2}=\alpha \left[ r\left( m_{1}\right) ,r\left( m_{2}\right) %
\right] +\beta \left[ r\left( m_{2}\right) ,p\left( m_{1}\right) \right]
-\beta \left[ r\left( m_{1}\right) ,p\left( m_{2}\right) \right] .
\end{equation*}

In the case 1 we assume that $\beta \neq 0$.

The Jacobi identity%
\begin{equation}
J\left( m_{1},m_{2},m_{3}\right) =\left( m_{1}\times m_{2}\right) \times
m_{3}+\left( m_{2}\times m_{3}\right) \times m_{1}+\left( m_{3}\times
m_{1}\right) \times m_{2}=0  \label{Jacobi}
\end{equation}%
must fulfill in $F\left( m_{1},m_{2},m_{3}\right) $.%
\begin{equation*}
\left( m_{1}\times m_{2}\right) \times m_{3}=
\end{equation*}%
\begin{equation*}
\left( \alpha \left[ r\left( m_{1}\right) ,r\left( m_{2}\right) \right]
+\beta \left[ r\left( m_{2}\right) ,p\left( m_{1}\right) \right] -\beta %
\left[ r\left( m_{1}\right) ,p\left( m_{2}\right) \right] \right) \times
m_{3}=
\end{equation*}%
\begin{equation*}
\alpha \left[ \alpha \left[ r\left( m_{1}\right) ,r\left( m_{2}\right) %
\right] ,r\left( m_{3}\right) \right] +
\end{equation*}%
\begin{equation*}
\beta \left[ r\left( m_{3}\right) ,\beta \left[ r\left( m_{2}\right)
,p\left( m_{1}\right) \right] -\beta \left[ r\left( m_{1}\right) ,p\left(
m_{2}\right) \right] \right] -
\end{equation*}%
\begin{equation*}
\beta \left[ \alpha \left[ r\left( m_{1}\right) ,r\left( m_{2}\right) \right]
,p\left( m_{3}\right) \right] =
\end{equation*}%
\begin{equation*}
\alpha ^{2}\left[ \left[ r\left( m_{1}\right) ,r\left( m_{2}\right) \right]
,r\left( m_{3}\right) \right] +\beta ^{2}\left[ r\left( m_{3}\right) ,\left[
r\left( m_{2}\right) ,p\left( m_{1}\right) \right] \right] -
\end{equation*}%
\begin{equation*}
\beta ^{2}\left[ r\left( m_{3}\right) ,\left[ r\left( m_{1}\right) ,p\left(
m_{2}\right) \right] \right] -\beta \alpha \left[ \left[ r\left(
m_{1}\right) ,r\left( m_{2}\right) \right] ,p\left( m_{3}\right) \right] .
\end{equation*}%
\begin{equation*}
F\left( m_{1},m_{2},m_{3}\right) =L\left( r\left( m_{1}\right) ,r\left(
m_{2}\right) ,r\left( m_{3}\right) \right) \oplus
\end{equation*}%
\begin{equation*}
\left( \bigoplus\limits_{i=1}^{3}A\left( r\left( m_{1}\right) ,r\left(
m_{2}\right) ,r\left( m_{3}\right) \right) p\left( m_{i}\right) \right) ,
\end{equation*}%
so $J\left( m_{1},m_{2},m_{3}\right) =\sum\limits_{i=0}^{3}J_{i}$, where 
\begin{equation*}
J_{0}\in L\left( r\left( m_{1}\right) ,r\left( m_{2}\right) ,r\left(
m_{3}\right) \right) ,
\end{equation*}%
\begin{equation*}
J_{i}\in A\left( r\left( m_{1}\right) ,r\left( m_{2}\right) ,r\left(
m_{3}\right) \right) p\left( m_{i}\right) ,
\end{equation*}%
$i=1,2,3$ and must fulfill $J_{i}=0$, $i=0,\ldots ,3$.%
\begin{equation*}
J_{1}=\beta ^{2}\left[ r\left( m_{3}\right) ,\left[ r\left( m_{2}\right)
,p\left( m_{1}\right) \right] \right] -\beta ^{2}\left[ r\left( m_{2}\right)
,\left[ r\left( m_{3}\right) ,p\left( m_{1}\right) \right] \right] -
\end{equation*}%
\begin{equation*}
\beta \alpha \left[ \left[ r\left( m_{2}\right) ,r\left( m_{3}\right) \right]
,p\left( m_{1}\right) \right] =
\end{equation*}%
\begin{equation*}
\beta ^{2}\left[ r\left( m_{3}\right) ,\left[ r\left( m_{2}\right) ,p\left(
m_{1}\right) \right] \right] -\beta ^{2}\left[ r\left( m_{2}\right) ,\left[
r\left( m_{3}\right) ,p\left( m_{1}\right) \right] \right] -
\end{equation*}%
\begin{equation*}
\beta \alpha \left[ r\left( m_{2}\right) ,\left[ r\left( m_{3}\right)
,p\left( m_{1}\right) \right] \right] +\beta \alpha \left[ r\left(
m_{3}\right) ,\left[ r\left( m_{2}\right) ,p\left( m_{1}\right) \right] %
\right]
\end{equation*}%
by (\ref{Maction}) and definition of representation of Lie algebra. So $%
\beta ^{2}+\beta \alpha =0$ must fulfill and we have that $\beta =-\alpha $, 
$\alpha \neq 0$. It is easy to check that $\beta =-\alpha $ enough for (\ref%
{Jacobi}). Therefore 
\begin{equation}
m_{1}\times m_{2}=\alpha \left( \left[ r\left( m_{1}\right) ,r\left(
m_{2}\right) \right] +\left[ r\left( m_{1}\right) ,p\left( m_{2}\right) %
\right] -\left[ r\left( m_{2}\right) ,p\left( m_{1}\right) \right] \right)
=\alpha \left[ m_{1},m_{2}\right]  \label{lie_case1}
\end{equation}%
by (\ref{new_brackets}) and (\ref{Maction}).

In the case 2, if $\beta =0$ we have that 
\begin{equation}
m_{1}\times m_{2}=\alpha \left[ r\left( m_{1}\right) ,r\left( m_{2}\right) %
\right] .  \label{lie_case2}
\end{equation}
If $\alpha =0$, then $F\left( m_{1},m_{2}\right) \times F\left(
m_{1},m_{2}\right) =\left\{ 0\right\} $, but $\left[ F\left(
m_{1},m_{2}\right) ,F\left( m_{1},m_{2}\right) \right] \neq \left\{
0\right\} $. By condition Op2) $F\left( m_{1},m_{2}\right) \cong \left(
F\left( m_{1},m_{2}\right) \right) _{W}^{\ast }$. From this contradiction we
conclude that hear also $\alpha \neq 0$.

We denote by $\mathfrak{p}$ the operation defined by the word $w_{p}\in
F\left( m\right) $. 
\begin{equation*}
\mathfrak{p}\left( m\right) =\delta r\left( m\right) +q_{\mathfrak{p}}\left(
r\left( m\right) \right) p\left( m\right) ,
\end{equation*}%
where $\delta \in k$, $q_{\mathfrak{p}}\in k\left[ x\right] $. $\mathfrak{p}%
\left( \lambda \ast m\right) =\lambda \ast \mathfrak{p}\left( m\right) $
must fulfill for every $\lambda \in k$. 
\begin{equation}
\lambda \ast \mathfrak{p}\left( m\right) =\lambda \delta r\left( m\right)
+\lambda q_{\mathfrak{p}}\left( r\left( m\right) \right) p\left( m\right) .
\label{proj1}
\end{equation}%
\begin{equation}
\mathfrak{p}\left( \lambda \ast m\right) =\delta r\left( \lambda \ast
m\right) +q_{\mathfrak{p}}\left( r\left( \lambda \ast m\right) \right)
p\left( \lambda \ast m\right) =\delta \lambda r\left( m\right) +q_{\mathfrak{%
p}}\left( \lambda r\left( m\right) \right) \lambda p\left( m\right) .
\label{proj2}
\end{equation}%
As above we decompose $q_{\mathfrak{p}}$ by homogeneous components according
the degree of $x$ and conclude as above by comparison of (\ref{proj1}) and (%
\ref{proj2}) that $\deg q_{\mathfrak{p}}=0$ and 
\begin{equation*}
\mathfrak{p}\left( m\right) =\delta r\left( m\right) +\varepsilon p\left(
m\right)
\end{equation*}%
where $\varepsilon \in k$.

$\mathfrak{p}\left( \mathfrak{p}\left( m\right) \right) =\mathfrak{p}\left(
m\right) $ must fulfill in $F\left( m\right) $. 
\begin{equation*}
\mathfrak{p}\left( \mathfrak{p}\left( m\right) \right) =\delta r\left(
\delta r\left( m\right) +\varepsilon p\left( m\right) \right) +\varepsilon
p\left( \delta r\left( m\right) +\varepsilon p\left( m\right) \right)
=\delta ^{2}r\left( m\right) +\varepsilon ^{2}p\left( m\right) .
\end{equation*}%
Therefore $\delta ^{2}=\delta $, $\varepsilon ^{2}=\varepsilon $.

If $\delta =\varepsilon =1$ then $\mathfrak{p}\left( m\right) =r\left(
m\right) +p\left( m\right) =m$ and $\mathfrak{p}\left( \left( F\left(
m\right) \right) _{W}^{\ast }\right) =\left( F\left( m\right) \right)
_{W}^{\ast }$ but $p\left( F\left( m\right) \right) \neq \left( F\left(
m\right) \right) $ contrary to $F\left( m\right) \cong \left( F\left(
m\right) \right) _{W}^{\ast }$. So it is impossible that $\delta
=\varepsilon =1$.

If $\delta =\varepsilon =0$, then $\mathfrak{p}\left( m\right) =0$ and $%
\mathfrak{p}\left( \left( F\left( m\right) \right) _{W}^{\ast }\right) =0$
but $p\left( F\left( m\right) \right) \neq 0$ contrary to $F\left( m\right)
\cong \left( F\left( m\right) \right) _{W}^{\ast }$. As above we conclude
that $\delta =\varepsilon =0$ is impossible.

If $\delta =1$, $\varepsilon =0$. Then $\mathfrak{p}\left( m\right) =r\left(
m\right) $, i.e. $\mathfrak{p}=r$. $\mathfrak{p}$ must be a derivation of $%
\left( F\left( m\right) \right) _{W}^{\ast }$. In the case 2, by (\ref%
{lie_case2}), we have that 
\begin{equation*}
\mathfrak{p}\left( m_{1}\times m_{2}\right) =r\left( \alpha \left[ r\left(
m_{1}\right) ,r\left( m_{2}\right) \right] \right) =\alpha \left[ r\left(
m_{1}\right) ,r\left( m_{2}\right) \right] ,
\end{equation*}%
\begin{equation*}
\mathfrak{p}\left( m_{1}\right) \times m_{2}+m_{1}\times \mathfrak{p}\left(
m_{2}\right) =r\left( m_{1}\right) \times m_{2}+m_{1}\times r\left(
m_{2}\right) =
\end{equation*}%
\begin{equation*}
\alpha \left[ rr\left( m_{1}\right) ,r\left( m_{2}\right) \right] +\alpha %
\left[ r\left( m_{1}\right) ,rr\left( m_{2}\right) \right] =2\alpha \left[
r\left( m_{1}\right) ,r\left( m_{2}\right) \right] .
\end{equation*}%
$char\left( k\right) =0$, so $\mathfrak{p}$ is not a derivation. In the case
1, by (\ref{lie_case1}), we have that 
\begin{equation*}
\mathfrak{p}\left( m_{1}\times m_{2}\right) =r\left( \alpha \left( \left[
r\left( m_{1}\right) ,r\left( m_{2}\right) \right] +\left[ r\left(
m_{1}\right) ,p\left( m_{2}\right) \right] -\left[ r\left( m_{2}\right)
,p\left( m_{1}\right) \right] \right) \right) =
\end{equation*}%
\begin{equation*}
\alpha \left[ r\left( m_{1}\right) ,r\left( m_{2}\right) \right] ,
\end{equation*}%
\begin{equation*}
\mathfrak{p}\left( m_{1}\right) \times m_{2}+m_{1}\times \mathfrak{p}\left(
m_{2}\right) =r\left( m_{1}\right) \times m_{2}+m_{1}\times r\left(
m_{2}\right) =
\end{equation*}%
\begin{equation*}
\alpha \left( \left[ r\left( r\left( m_{1}\right) \right) ,r\left(
m_{2}\right) \right] +\left[ r\left( r\left( m_{1}\right) \right) ,p\left(
m_{2}\right) \right] -\left[ r\left( m_{2}\right) ,p\left( r\left(
m_{1}\right) \right) \right] \right) +
\end{equation*}%
\begin{equation*}
\alpha \left( \left[ r\left( m_{1}\right) ,r\left( r\left( m_{2}\right)
\right) \right] +\left[ r\left( m_{1}\right) ,p\left( r\left( m_{2}\right)
\right) \right] -\left[ r\left( r\left( m_{2}\right) \right) ,p\left(
m_{1}\right) \right] \right) =
\end{equation*}%
\begin{equation*}
\alpha \left( \left[ r\left( m_{1}\right) ,r\left( m_{2}\right) \right] +%
\left[ r\left( m_{1}\right) ,p\left( m_{2}\right) \right] \right) +\alpha
\left( \left[ r\left( m_{1}\right) ,r\left( m_{2}\right) \right] -\left[
r\left( m_{2}\right) ,p\left( m_{1}\right) \right] \right) =
\end{equation*}%
\begin{equation*}
\alpha \left( 2\left[ r\left( m_{1}\right) ,r\left( m_{2}\right) \right] +%
\left[ r\left( m_{1}\right) ,p\left( m_{2}\right) \right] -\left[ r\left(
m_{2}\right) ,p\left( m_{1}\right) \right] \right) .
\end{equation*}%
In this case $\mathfrak{p}$ also is not a derivation.

Therefore we have only one possibility: $\delta =0$, $\varepsilon =1$. It
means $\mathfrak{p}\left( m\right) =p\left( m\right) $, i.e., $\mathfrak{p}%
=p $.

And in the case 2, by (\ref{lie_case2}), we have that 
\begin{equation*}
p\left( F\left( m_{1},m_{2}\right) \times F\left( m_{1},m_{2}\right) \right)
=0
\end{equation*}%
but 
\begin{equation*}
p\left[ r\left( m_{1}\right) ,p\left( m_{2}\right) \right] =\left[ r\left(
m_{1}\right) ,p\left( m_{2}\right) \right] \neq 0,
\end{equation*}%
so 
\begin{equation*}
p\left[ F\left( m_{1},m_{2}\right) ,F\left( m_{1},m_{2}\right) \right] \neq 0
\end{equation*}%
contrary to $F\left( m_{1},m_{2}\right) \cong \left( F\left(
m_{1},m_{2}\right) \right) _{W}^{\ast }$. Therefore the case 2 is impossible.

Hence 
\begin{equation*}
m_{1}\times m_{2}=\alpha \left( \left[ r\left( m_{1}\right) ,r\left(
m_{2}\right) \right] +\left[ r\left( m_{1}\right) ,p\left( m_{2}\right) %
\right] -\left[ r\left( m_{2}\right) ,p\left( m_{1}\right) \right] \right)
=\alpha \left[ m_{1},m_{2}\right] ,
\end{equation*}%
where $\alpha \neq 0$.

From this fact, as in \cite[end of the subsection 2.5]{PlotkinZhitAutCat},
we conclude that $\Psi \in \mathfrak{Y}$. So $\mathfrak{S=S\cap Y}$ and $%
\mathfrak{A/Y=}\left\{ 1\right\} $.
\end{proof}

\section{The main theorem.}

\setcounter{equation}{0}

\begin{theorem}
If $\mathrm{Aut}k=\left\{ id_{k}\right\} $ then automorphic equivalence of
representations of Lie algebras coincides with the geometric equivalence.
\end{theorem}

\begin{proof}
We assume that $H_{1}=\left( L_{1},V_{1}\right) ,H_{2}=\left(
L_{2},V_{2}\right) \in \Xi $ are automorphically equivalent. By Theorem \ref%
{automorphEquiv} we have that $N_{1}=\mathcal{F}\left( H_{1}\right) ,N_{2}=%
\mathcal{F}\left( H_{2}\right) $ are automorphically equivalent. By \cite[%
Proposition 9]{PlotkinSame} and Theorem \ref{group} we can conclude from
this fact that $N_{1},N_{2}$ are geometrically equivalent. It means that $%
Cl_{N_{1}}\left( F\right) =Cl_{N_{2}}\left( F\right) $ for every $F\in 
\mathrm{Ob}\Theta ^{0}$.

We will consider the arbitrary $W_{1}=\left( L\left( X_{1}\right) ,A\left(
X_{1}\right) Y_{1}\right) \in \mathrm{Ob}\Xi ^{0}$. There are $X_{2}\subset
X^{0}$, $Y_{2}\subset Y^{0}$ such that $X_{1}\subseteq X_{2}$, $%
Y_{1}\subseteq Y_{2}$ and $W_{2}=\left( L\left( X_{2}\right) ,A\left(
X_{2}\right) Y_{2}\right) \in \Xi ^{\prime }$. By Theorem \ref%
{revFreeWithProj} there exists $F\in \mathrm{Ob}\Theta ^{0}$ such that $F=%
\mathcal{F}\left( W_{2}\right) $. By Proposition \ref{ClCorresp} we can
conclude from $Cl_{N_{1}}\left( F\right) =Cl_{N_{2}}\left( F\right) $ that $%
Cl_{H_{1}}\left( W_{2}\right) =Cl_{H_{2}}\left( W_{2}\right) $. And by
Theorem \ref{cl} we can conclude that $Cl_{H_{1}}\left( W_{1}\right)
=Cl_{H_{2}}\left( W_{1}\right) $. So $H_{1}$ and $H_{2}$ are geometrically
equivalent.
\end{proof}

\section{Acknowledgements.}

We acknowledge the support by FAPESP - Funda\c{c}\~{a}o de Amparo \`{a}
Pesquisa do Estado de S\~{a}o Paulo (Foundation for Support Research of the
State S\~{a}o Paulo), projects No. 2010/50948-2 and No. 2010/50347-9.


\begin{thebibliography}{9}
\bibitem{PlotkinSame} B.Plotkin, Algebras with the same algebraic geometry, 
\textit{Proceedings of the Steklov Institute of Mathematics, MIAN,} \textbf{%
242, }(2003), pp. 176--207.

\bibitem{Repofgr} B. Plotkin, A. Tsurkov, Action type geometrical
equivalence of representations of groups. \textit{Algebra and Discrete
Mathematics,} 4 (2005), pp. 48 - 79.

\bibitem{PlotkinVovsi} Plotkin B.I., Vovsi, S.M. Varieties of Group
Representation, \textit{Zinatne.} Riga, 1983, (Russian).

\bibitem{PlotkinZhitAutCat} B. Plotkin, G. Zhitomirski, On automorphisms of
categories of free algebras of some varieties, \textit{Journal of Algebra,} 
\textbf{306:2}, (2006), pp. 344 -- 367.

\bibitem{TsurkovAutomEquiv} A. Tsurkov, Automorphic equivalence of algebras. 
\textit{International Journal of Algebra and Computation.} \textbf{17:5/6},
(2007), pp. 1263--1271.
\end{thebibliography}
\end{document}